\documentclass[a4paper,10pt,reqno]{amsart}
\usepackage[T1]{fontenc}
\usepackage[utf8]{inputenc}
\usepackage{xcolor}
\usepackage{hyperref}
\usepackage{amssymb}
\usepackage{eucal}
\usepackage{a4wide}
\usepackage{graphics}
\usepackage{graphicx}
\usepackage{xypic}
\usepackage{amsmath,mathrsfs,mathtools}
\usepackage{wasysym}
\usepackage{dsfont}
\usepackage{float}
\usepackage{fancyhdr}
\usepackage{comment}
\usepackage[labelfont=bf]{caption}
\usepackage{algorithmicx}
\usepackage{algorithm}
\usepackage{algpseudocode}
\usepackage{xpatch}
\xpatchcmd{\paragraph}{\normalfont}{{\normalfont\itshape}}{}{}

\newtheorem{assumption}{Assumption}[section]
\newtheorem{theorem}{Theorem}[section]
\newtheorem{lemma}{Lemma}[section]
\newtheorem{remark}{Remark}[section]
\theoremstyle{definition}
\newtheorem{definition}{Definition}[section]
\graphicspath{ {./Pictures/} }
\newcommand{\lp}{\left(}
\newcommand{\rp}{\right)}
\newcommand{\der}{\partial}
\newcommand{\R}{\mathbb{R}}      

\newcommand{\N}{\mathbb{N}}      

\newcommand{\Flder}{\rightarrow}
\newcommand{\dd}{\mathrm{d}}     
\renewcommand{\Lc}{\mathcal{Lc}^{\text{con}}}
\newcommand{\Lf}{\mathcal{L}^{\text{fra}}}
\renewcommand{\L}{\mathcal{L}}
\DeclareMathOperator{\diag}{diag}
\usepackage[
  sorting=nty,giveninits=true
]{biblatex}

\DeclareFieldFormat[article, book,incollection]{volume}{\mkbibbold{#1}}
\DeclareFieldFormat[article]{number}{\mkbibbold{\mkbibparens{#1}}}
\DeclareFieldFormat[book]{number}{\mkbibbold{#1}}
\DeclareFieldFormat[article,book,incollection]{year}{\mkbibbold{\mkbibparens{#1}}}
\DeclareFieldFormat{pages}{pp.~#1}
\DeclareBibliographyDriver{article}{%
  \printnames{author}%
  \adddot\space
  \printfield[title]{title}
  \adddot\space
  \printfield{journaltitle}
  \setunit{\addspace}%
  \printfield{volume}
  \iffieldundef{number}{}{%
    \printfield{number}
  }%
  \iffieldundef{pages}{}{%
    ,\addspace\printfield{pages}
  }%
 \space
  \printfield{year}
  \finentry
}
\DeclareBibliographyDriver{book}{%
  \printnames{author}
  \setunit{\addcolon\space}%
  \printfield{title}
    \newunit
  \printfield{edition}
  \newunit
  \printfield{series}
  \iffieldundef{volume}{}{, vol.~\printfield{volume}}
  \newunit
  \printlist{publisher}
  \setunit{\addcomma\space}%
  \printlist{location}
  \setunit{\addspace}%
  \printfield{year}
  \finentry
}
\DeclareBibliographyDriver{incollection}{%
  \printnames{author}
  \adddot\space
  \printfield[title]{title}
  \adddot\space
  \printtext{In: }%
  \printfield{booktitle}
  \iffieldundef{editor}{}{%
    \addcomma\space
    \printnames{editor}\addspace(ed.)
  }%
  \iffieldundef{series}{}{%
    \addcomma\space\printfield{series}
    \iffieldundef{volume}{}{, vol.~\printfield{volume}}%
  }%
  \iffieldundef{pages}{}{,\space\printfield{pages}}
  \newunit
  \printlist{publisher}
  \setunit{\addcomma\space}%
  \printlist{location}
  \addcomma\space
  \printfield{year}
  \finentry
}

\begin{document}
\title{On Runge-Kutta convolution quadrature based fractional variational integrators}

\begin{abstract}
Lagrangian systems subject to fractional damping can be incorporated into a variational formalism. 
The construction can be made by doubling the state variables and introducing fractional derivatives \cite{JiOb2}. The main objective of this paper is to use the Runge–Kutta convolution quadrature (RKCQ) method for approximating fractional derivatives, combined with higher order Galerkin methods in order to derive fractional variational integrators (FVIs). We are specially interested in the CQ based on Lobatto IIIC. Preservation properties such as energy decay as well as convergence properties are investigated numerically and proved for 2nd order schemes. The presented schemes reach 2nd, 4th and 6th  accuracy order. A brief discussion on the midpoint fractional integrator is also included.
\end{abstract}



\author{Khaled Hariz Belgacem}
\email{hariz@math.upb.de}
\address{Department of Mathematics, University of Paderborn, Warburger Straße 100, 33098 Paderborn, Germany}

\author{Sina Ober-Bl\"obaum}
\email{sinaober@math.upb.de}
\address{Department of Mathematics, University of Paderborn, Warburger Straße 100, 33098 Paderborn, Germany}

\email{fjimenez@ind.uned.es}
\author{Fernando Jim\'enez}
\address{Departamento de Matemática Aplicada I, ETSII, Universidad Nacional de Educación a Distancia (UNED)
c. Juan del Rosal 12, 28040, Madrid, Spain.}

\maketitle

\section{Introduction}
Differential equations are extensively used to model many phenomena in science and engineering, however, they often might not be capable in describing systems having long-term memory property due to their local nature. For instance, phenomena like anomalous diffusion, deformations, viscoelastic materials \cite{Bagley}, biological processes \cite{Magin1, Magin2} and other applications in different fields 
\cite{Bonilla, Hilfer} require models that incorporate historical effects.   In this context, fractional differential equations, which involve fractional derivatives, provide a powerful extension. Their nonlocal nature allows the current behavior of a system to depend not only on the present but also on its entire past history, leading to more accurate and realistic representations of complex phenomena. 
We refer to \cite{Oldham,TheBook,Ignor,TheBook2} for a comprehensive overview of the fractional calculus theory. Although fractional calculus has many applications, developing analytical solutions for fractional differential equations remains challenging. Some modern applications of fractional calculus to physical systems, complex media, quantum dynamics and more can be found in \cite{Tarasov}. Applications in mechanics such as non-local elasticity,
viscoelasticity models, heat conduction diffusion problems are treated in \cite{chala}.
See also \cite{HandbookCon,HandbookEng} for applications in control, engineering, life and social sciences. 

Recently, fractional calculus has emerged as a powerful tool in the study of dynamical systems. Its applications extend across a wide range of models, including dissipative and conservative systems as well as constrained dynamical systems \cite{Riewe,Riewe2,Klimek,Agra,Baleanu}. Of particular interest in this work are fractional variational problems, which extend the classical variational framework to the fractional setting. Such problems naturally give rise to fractional equations of motion, namely the fractional Euler–Lagrange equations, which generalize their classical counterparts.

Classical variational principles fail to capture dissipative systems \cite{Bauer}, fractional calculus provides a natural extension of the variational framework, enabling the dissipation to be fully treated within the framework of variational principles.
This topic has been discussed by several authors \cite{Riewe, Cresson2,Cresson3, JiOb1,JiOb2}. In particular, Lagrangian systems subject to linear  damping with constant coefficients have been discussed in \cite{JiOb2} by means of restricted Hamilton's principle. This approach shows that the dynamical equations are sufficient but not necessary conditions for characterizing the extremals of the action.

At the discrete level, the restricted Hamilton's principle gives rise to Fractional Variational Integrators by employing two different approximations: one for the conservative part and the other for the fractional one. A well known discretization for the conservative Lagrangian can be found in \cite{MaWe, LuGeWa}. 
Although several approaches exist for approximating the fractional part \cite{Lidia2015,BoXuHu,HandbookNum,Kaibook2016}, the chosen method must ensure that essential properties are preserved in the discrete setting, as in the continuous one.

In this context, FVIs have been derived in \cite{JiOb2} using the Grünwald-Letnikov approximation for the fractional derivatives which  leads to first-order accurate schemes. More recently, \cite{HaJiOb} proposed a generalization by employing convolution quadrature (CQ) \cite{Lubich2} for the fractional derivatives, which can in principle achieve approximations up to order $6$ with fractional backward difference (BDF) methods,  and Galerkin variational methods for the conservative part \S\ref{HO-Action}. Nevertheless, the overall accuracy of FVIs remains limited to second order. This restriction due to saturation effects of BDFCQ. In addition, BDFCQ neglects the inner nodes that used in higher-order approximations of the conservative action, see \cite{HaJiOb} more details. This limitation can be effectively addressed by employing Runge–Kutta convolution quadrature (RKCQ), which naturally avoids the saturation effects of BDFCQ and retains higher-order accuracy.

The aim of this paper is to improve the accuracy of FVI by employing RKCQ, see \cite{LuOs, BaLu,BaLo} for a deeper discussion about the construction and the analysis concerning Radau IIA or Lobatto IIIC methods, and \cite{LeMa} for Gauss methods. This large class of methods enables higher order approximation for fractional derivatives. We derive Lobatto IIIC FVI and the validity is tested by presenting analytically solvable models and comparing the simulation results with their exact solutions.

\paragraph{Main contributions} In this paper, we present new fractional variational integrators derived from a strategy analogous to \cite{JiOb2} but employing different discretization schemes. The primary innovation lies in our treatment of the fractional part, leading to the following main contributions:
\begin{enumerate}
\item We establish the semigroup and  asymmetric integration by parts properties for a general RKCQ.
\item We demonstrate how to use CQ based on Runge-Kutta Lobatto IIIC methods  for the approximation of fractional part of our Lagrangian.
\item We prove the approximation order for a specific example with quadratic
Lagrangian.
\item We exanimate numerically the accuracy of our method for the Bagley-Torfik equation which involves the half-derivative.
\item  We establish the midpoint RKCQ, which gives rise to the midpoint fractional integrators (second order accuracy).
\end{enumerate}

\paragraph{Structure of the paper} 
Section \ref{Disc_Lag_Mech} contains a brief summary of Lagrangian formalism and  variational integrators, which are used throughout this work. In Section \ref{FracIntegrals}, we review some of the standard definitions and properties on fractional operators and  their approximations. In particular, Subsection \ref{sec:RKCQ} introduces the notion of Runge-Kutta based convolution quadrature by Lubich \cite{LuOs}.
Section \ref{sec:restrectedCOV} presents the continuous restricted Hamilton's principle, followed by an explanation of how RKCQ can be applied in the discrete setting to derive discrete dynamics. We state in Theorem \ref{thm:maintheorem} a new scheme associated to Lobatto IIIC (Table \ref{tab:Lobatto}). The energy decay and accuracy of the obtained schemes are numerically examined and, in particular, theoretically proven to be of order 2. A brief discussion of midpoint CQ is also included. Section \ref{conclusion} concludes the paper.

\section{Discrete Lagrangian Mechanics}\label{Disc_Lag_Mech}

In this section, we give a brief summary of the Lagrangian mechanical system from both continuous and  discrete points of view. 

\subsection{Lagrangian Hamilton's Principle}
Consider mechanical system with a $d$-dimensional configuration manifold $Q$, whose motion is described by a trajectory in the tangent bundle $TQ$ associated to $Q$, parameterized by positions 
$q(t)\in Q$ and   velocities  $\dot q(t)\in T_{q(t)}Q$. The Lagrangian system is characterized by a $C^2$-Lagrangian $L:TQ\Flder\R$, commonly given by  $L(q,\dot q)=T(q,\dot q)-V(q)$, where  $T$ and $U$ denote the kinetic and potential energies, respectively. Consider 
the action $\mathcal{L}:C^2([0,T],Q)\Flder\R$ of a mechanical system, given by
\begin{equation}\label{ContAc}
\mathcal{L}(q)=\int_0^TL(q(t),\dot q(t))\,\dd t.
\end{equation}
In Lagrangian mechanics, the trajectory is characterized by Hamilton's principle.  That is, the trajectory $q$, with fixed points $q(0)$ and $q(T)$, is determined by saying the action is stationary, namely
\[
\delta\mathcal{L}(q)=0.
\]
for all arbitrary variations that satisfy $\delta q(0)=\delta q(T)=0$. Necessary and sufficient conditions for the extremals lead to the Euler-Lagrange equations
\begin{equation}\label{EL}
\frac{\dd}{\dd t}\lp\frac{\der L}{\der\dot q}(q(t),\dot q(t))\rp-\frac{\der L}{\der q}(q(t),\dot q(t))=0.
\end{equation}
These are the equations of motion for conservative system. See \cite{AbMa,MaWe,MarRat} for more details.

\subsection{Discrete Lagrangian Hamilton's Principle}\label{Disc}
Discrete mechanics is essentially a discrete formulation of geometric mechanics that preserves its underlying structure within a numerical framework, leading to what are also called \textit{structure-preserving methods}.  Variational integrators (VIs) are a discrete formulation of the Euler–Lagrange equation \eqref{EL} and belong to a family of structure-preserving methods. Unlike standard discretizations of differential equations,  the construction of VIs starts by discretizing  the state space $TQ$ and the action $\mathcal{L}$ \cite{MaWe}. More concretely, consider a time discretization $t_k=\left\{kh\right\}_{k=0}^N$ of the interval $[0,T]$ where $h=T/N$ is the time step. One approximates the continuous trajectory $q(t)$ by a discrete sequence $q_d:=\left\{ q_k\right\}_{0:N}\in Q^{N+1}$ at times $t_k$, i.e.,~$q_k\approx q(t_k)$. The discrete Lagrangian $L_d$ has to be defined in such a way as to approximate the action over a short time interval.  Thus, define $L_d:Q\times Q\Flder\R$ based on two neighboring configurations $q_k$ and $q_{k+1}$ as
\begin{equation}\label{TwoPointsApprox}
L_d(q_k,q_{k+1})\approx \int_{t_k}^{t_{k+1}}L(q(t),\dot q(t))\,\dd t.
\end{equation}
Furthermore, the discrete action sum $\mathcal{L}_d:Q^{N+1}\Flder\R$ approximates the continuous one \eqref{ContAc} as
\[
\mathcal{L}_d(q_d)=\sum_{k=0}^{N-1}L_d(q_k,q_{k+1}).
\]
The discrete variational principle is designed to mirror the continuous one. Namely, the trajectory $q_d=\{q_k\}_{k=0}^N$, with fixed points $q_0$ and $q_N$, is determined by saying the action sum is stationary
\[
\delta \mathcal{L}_d(q_d)=0,
\]
for arbitrary variations $\delta q_0=\delta q_N=0$, leading to the discrete Euler-Lagrange (DEL) equation 
\begin{equation}\label{DEL}
D_2L_d(q_{k-1},q_{k})+D_1L_d(q_k,q_{k+1})=0,\qquad k=1,\ldots, N-1,
\end{equation}
where $D_i$ refers to the partial derivative with respect to the $i$-th argument. This scheme is called variational integrator and its flow is defined implicitly by the map  $\Phi_h:Q\times Q\Flder Q\times Q$, $(q_{k-1},q_k)\mapsto (q_k,q_{k+1})$. 
The resulting integrator \eqref{DEL} belongs to the class of structure-preserving methods \cite{LuGeWa,Raff2023,SerFer2016}, meaning that it preserves, or nearly preserves, the invariants of the underlying continuous system such as, the symplectic structure, the total energy and, the momentum in presence of symmetries. Despite this advantage, this construction yields at most a second-order accurate scheme. The purpose of the following paragraph is to introduce a family of variational methods of arbitrary order. 

\subsection{Galerkin Variational Methods}\label{HO-Action}
We now consider another class of methods to generate variational integrators, known as Galerkin methods. This approach involves locally interpolating the trajectories on a finite-dimensional basis and evaluating the action integral using a suitable quadrature method \cite{HaLe13,MaWe,SinaSaake}.

\paragraph{(1) Trajectories space} We replace the infinite-dimensional function space
$$\mathcal{C}([t_k,t_{k+1}],Q)=\left\{ q:[t_k,t_{k+1}]\Flder Q\,|\, q(t_k)=q_k,\,q(t_{k+1})=q_{k+1} \right\},$$
by a finite-dimensional one  $\mathcal{C}^s([t_k,t_{k+1}],Q)\subset \mathcal{C}([t_k,t_{k+1}],Q)$, where $\mathcal{C}^s([t_k,t_{k+1}],Q)$ denotes the space of polynomials of degree $s$.  Given $s+1$ control points $0=d_0<d_1<\cdots<d_{s-1}<d_s=1$ and $s+1$ configurations $q_k=(q_k^0,q_k^1,\ldots,q_k^{s-1},q_k^s)$, which satisfy $q_k^0=q_k$ and $q_k^{s}=q_{k+1}$. Thus, there exists a basis $\{\ell_\nu\}_{\nu=0}^s$, e.g. the Lagrange polynomials, for which an element $q_d\in \mathcal{C}^s([t_k,t_{k+1}],Q)$ can be uniquely written as   a  linear combination of $\{\ell_\nu\}_{\nu=0}^s$ 
\begin{equation}\label{Polynomials}
q_d(t;k):=\sum_{\nu=0}^sq_k^{\nu}\,\ell_{\nu}\left(\frac{t}{h}\right),
\end{equation}
 such that $\ell_{\nu}(d_i)= \delta_{\nu i}$ and therefore $q_d(d_\nu h;k)=q_k^\nu$ according to \eqref{Polynomials}.  A trajectory $q_d : [0, T ] \to Q$, defined over the entire interval, is expressed on each sub-interval using the local interpolation $q_d$. This gives, for all $t \in [t_k , t_{k+1}]$
\[q_d|_{[t_k , t_{k+1}]} (t) = q_d (t;k).\]
To ensure the continuity of $q_d$ at the main nodes, we write for all $0 < k < N$ the condition
\[q_d (t_k,k) = q_d (t_k,k-1)=q_k.\]

\paragraph{(2) Quadrature for the action integral} We first replace the curve $q(t)$ and the velocity $\dot q(t)$ by their polynomial counterparts $q_d(t;k)$, $\dot q_d(t;k)$ on the action as
\begin{equation}\label{QuadforL_d}
    \int_{t_k}^{t_{k+1}}L(q_d(t;k),\dot q_d(t;k))\,\dd t,\quad k=0,\ldots, N-1.
\end{equation}
Then, a quadrature rule $(b_i,c_i)_{i=1}^r$ with $c_i\in[0,1]$ is applied to \eqref{QuadforL_d} on the same time interval yields
\begin{equation}\label{DiscLag}
L_d(\{q_k\})\equiv L_d(q_k^0, \ldots,q_k^s):=h\sum_{i=1}^rb_iL(q_d(c_i\,h;k),\dot q_d(c_i\,h;k))\approx \int_{t_k}^{t_{k+1}}L(q(t),\dot q(t))\,\dd t.
\end{equation}
Following \S\ref{Disc}, the action sum can be defined from \eqref{DiscLag} by
\[
\mathcal{L}_d(q_d)=\sum_{k=0}^{N-1}L_d(\{q_k\}),
\]
and hence the discrete Hamilton's principle leads to the discrete Euler-Lagrange equations
\[
\begin{split}
D_{s+1}L_d(q_{k-1}^0,\ldots,q_{k-1}^s)+D_1L_d(q_{k}^0,\ldots,q_{k}^s)=0,\\
D_iL_d(q_{k}^0,\ldots,q_{k}^s)=0,\quad  \forall\,i=2,\ldots,s;
\end{split}
\]
for $k=1,\ldots,N-1$ where $D_iL_d(q_k^{0},\ldots,q_k^{s}):=\frac{\partial L_d(\{q_k\})}{\partial q_k^{i-1}}$. See \cite{MaWe, SinaSaake} for further details.

Variational integrators have been widely used for conservative systems \cite{HaLe13,Leok2011,SinaSaake,Ca13}, nonconservative systems \cite{SinaVerm,ShPaWo}, optimal control problems \cite{Ca14} and multirate systems \cite{SinaWeGaLe,WeSinaLe,LeSina}.

The discrete variational principle will automatically share the accuracy of the Lagrangian to the resulting DEL equation by means of what is called {\it local variational order} \cite{MaWe}, that is,  \textit{to construct  a geometric integrator of order $r$, focus on creating an approximation to the action integral of order $r+1$}. Depending on the degree of the interpolating polynomial and choice of the quadrature rule, one obtained  VIs of different computational complexity and accuracy.  Galerkin variational methods with polynomial interpolation of degree $s$ and quadrature of order $s$ were shown to converge with order at least $s$ \cite{HaLe13} and the optimal order is generally unknown. Numerical studies \cite{SinaSaake} indicate convergence of order $\min(2s,r)$ for Lobatto and Gauss quadrature rules.

\section{Fractional Operators and their Approximations}\label{FracIntegrals}

\subsection{Fractional Integrals and Fractional Derivatives}
In this section, we collect the definitions of fractional operators and their main properties without proofs. For more details see  \cite{TheBook2,TheBook}.

\begin{definition}[Fractional integrals]
Let $f:[0,T]\Flder\R$ be a continuous and integrable function. The left and right Riemann–Liouville fractional integrals of order $\alpha\in \R^+$ are defined respectively in the following way:   
\begin{subequations}\label{RLInt}
\begin{align}
J^{\alpha}_{-}f(t)&=\frac{1}{\Gamma(\alpha)}\int_0^t(t-\tau)^{\alpha-1}f(\tau)\,\dd\tau, \quad t\in (0,T],\label{RLInt:a}\\
J^{\alpha}_{+}f(t)&=\frac{1}{\Gamma(\alpha)}\int_t^T(\tau-t)^{\alpha-1}f(\tau)\,\dd\tau,\quad t\in [0,T), \label{RLInt:b}
\end{align}
\end{subequations}
where $\Gamma$ is the Euler's Gamma function and for $\alpha=0$,  we set $J^{0}_{-}f=J^{0}_{+}f=f$.
\end{definition}
The fractional integrals enjoy two significant properties which are crucial in application,   the {\it integration by parts} and  the {\it semigroup property}\footnote{Formula \eqref{IntegrationByParts} holds true  for $\alpha>0$, $f \in  L^p (a, b),\, g\in L^q (a, b)$ with $p \geq 1,\, q \geq 1$ and
$\frac{1}{p} + \frac{1}{q} \leq 1 + \alpha$, while formula \eqref{Aditive} holds almost everywhere in $[a,b]$ for $\alpha,\beta>0$ and $f \in  L^p (a, b)$ with
$1\leq p \leq\infty $.} \cite[Lemmas 2.3 and 2.7]{TheBook2}
\begin{subequations}\label{FracIntProperties}
\begin{align}
\int_0^Tf(t)\left(J^{\alpha}_{\sigma}g\right)(t)\, \dd t&=\int_0^T g(t) \left(J^{\alpha}_{-\sigma}f\right)(t)\, \dd t,\quad  \sigma\in\left\{-,+\right\}, \label{IntegrationByParts}\\
J_{\sigma}^{\alpha}J_{\sigma}^{\beta}f&=J_{\sigma}^{\alpha+\beta}f,\quad \alpha,\beta>0. \label{Aditive}
\end{align}
\end{subequations}

\begin{definition}[Fractional derivatives]
The left and right Riemann–Liouville fractional integrals of order $\alpha\in \R^+$  are defined respectively in the following way: 
\begin{subequations}\label{RLDer}
\begin{align}
D^{\alpha}_{-}f(t)&=\frac{1}{\Gamma(n-\alpha)}\left(\frac{\dd}{\dd t}\right)^n\int_0^t(t-\tau)^{n-\alpha-1}f(\tau)\,\dd\tau,\quad\,t\in (0,T],\label{RLDer:a}\\
D^{\alpha}_{+}f(t)&=\frac{1}{\Gamma(n-\alpha)}\left(-\frac{\dd}{\dd t}\right)^n\int_t^T(\tau-t)^{n-\alpha-1}f(\tau)\,\dd\tau,\quad t\in [0,T), \label{RLDer:b}
\end{align}
\end{subequations}
with $n \in \N$, $n -1\leq \alpha < n$.
\end{definition}
In particular, for $\alpha=n\in\N$, we recover the usual derivatives, namely
\[D^{n}_{-}f(t)=f^{(n)}(t),\qquad D^{n}_{+}f(t)=(-1)^nf^{(n)}(t).\]
In the variational problems, two important properties for fractional derivatives\footnote{Formula \eqref{IntegrationByParts2} holds true  for $\alpha>0$, $f \in  I_-^\alpha(L^p),\, g\in I_+^\alpha(L^q)$ with $p \geq 1,\, q \geq 1$, and
$\frac{1}{p} + \frac{1}{q} \leq 1 + \alpha$, see \cite{TheBook} for the definition of $I_-^\alpha(L^p)$ and $I_+^\alpha(L^p)$.} playing a significant role:
\begin{subequations}\label{FracDerProperties}
\begin{align}
\int_0^Tf(t)\left(D^{\alpha}_{\sigma}g\right)(t)\, \dd t&=\int_0^T g(t) \left(D^{\alpha}_{-\sigma}f\right)(t)\, \dd t,\quad  \sigma\in\left\{-,+\right\}, \label{IntegrationByParts2}\\
D_{\sigma}^{\alpha}D_{\sigma}^{\alpha}f&=D_{\sigma}^{(2\alpha)}f,\quad \alpha>0,\label{Aditive2}
\end{align}
\end{subequations}
It is important to note that, unlike fractional integrals which satisfy the semigroup property, fractional derivatives do not in general exhibit such property unless additional conditions are imposed, see \cite[Theorem 2.13]{Kai} and \cite{TheBook} for further details and proofs.

\begin{remark}
Following \cite[Lemma 2.2]{TheBook2}, the existence of the fractional derivatives requires $f\in AC^{n}([0,T])$, the functional space as defined in \cite[Equation (1.1.7)]{TheBook2} or \cite[Definition 1.3]{TheBook}. In this case, the left Riemann–Liouville fractional derivative can be rewritten  as
\begin{equation}\label{caputo}
 D^{\alpha}_{-}f(t)=\sum_{k=0}^{n-1}\frac{f^{(k)}(0)}{\Gamma(1+k-\alpha)}t^{k-\alpha}+\frac{1}{\Gamma(n-\alpha)}\int_0^t(t-\tau)^{n-\alpha-1}f^{(n)}(\tau)\,\dd\tau. 
\end{equation}
The second term in the right hand-side is usually referred to as the \textnormal{left Caputo fractional derivative} which means that, the Riemann–Liouville and the Caputo derivatives are identical if and only if $f^{(k)}(a) = 0,\ k = 0, \ldots , n - 1$.
\end{remark}

\subsection{Runge-Kutta Convolution Quadratures}\label{sec:RKCQ}
The theory of convolution quadrature methods (CQ) was developed in \cite{Lubich1,Lubich2,Lubich3}. It provides a systematic framework to compute  the convolution integral 
\begin{equation}\label{ContConv}
\left(\kappa\ast f\right)(t):=\int_0^t\kappa(t-\tau)f(\tau)\,\dd\tau, \quad t>0
\end{equation}
efficiently without evaluating the convolution kernel $\kappa$, but instead through its Laplace transform $K(s)$ by combining quadrature rules with linear multistep methods. A particular case of convolution is the {\it retarded} fractional integral \eqref{RLInt:a} where the  kernel $\kappa$ is given by
\begin{equation}\label{ConvKer}
\kappa(t)=\kappa^{(\alpha)}(t)=\frac{t^{\alpha-1}}{\Gamma(\alpha)}.
\end{equation}
with the corresponding Laplace transform of the form
\begin{equation}\label{LaplaceTrans}
K^{(\alpha)}(s):=\mathscr{L}(\kappa)(s)=\int_0^{\infty}\frac{t^{\alpha-1}}{\Gamma(\alpha)}e^{-st}\, \dd t=\frac{1}{s^\alpha}.
\end{equation}

Building on this theory, subsequent extensions have applied CQ with Runge-Kutta methods \cite{LuOs, BaLu,BaLo}, which enable higher-order accuracy. In the present work, we adopt this Runge-Kutta-based approach in order to construct higher-order fractional variational integrators in \S\ref{sec:FVI}. 

\subsubsection{Runge-Kutta methods}
We first recall some notations of Runge-Kutta (RK) methods which will be used throughout the following sections. We refer to \cite{HaWa,Raff2023}. Consider a $r$-stage RK method described by the coefficient matrix $A=(a_{ij})_{i,j=1}^r\in\R^{r\times r}$, the vector of weights $\mathbf{b}=(b_1,\ldots,b_r)^\top\in\R^r$ and the vector of abcissae $\mathbf{c}=(c_1,\ldots,c_r)^\top\in[0,1]^r$, or by means of the standard Butcher tableau
$$\setlength{\tabcolsep}{20pt}
\renewcommand{\arraystretch}{1.5}\begin{array}{c|c}
\mathbf{c}   &  A  \\ \hline
     &  \mathbf{b} ^\top 
\end{array}.$$
Let the time step $h = T /N$, and the grid points $t_j = jh$, an $r$-stage RK method applied to ordinary differential equations
\[u'(t)=f(t,u(t)),\quad u(0)=u_0\]
is given by 
\begin{subequations}\label{eq:RKmethod}
\begin{align}
U_{ki}&= z_k +h\sum_{j=1}^r a_{ij}f(t_k+c_jh,U_{kj}),\quad j=1,\ldots,r\\
u_{k+1}&= z_k +h\sum_{j=1}^r b_{j}f(t_k+c_jh,U_{kj}).
\end{align}
\end{subequations}
The method \eqref{eq:RKmethod} has \textit{classical order} $p\geq 1$ and \textit{stage order} $q \leq p$, respectively, if for sufficiently smooth $f$ (see  \cite{HaWa})
\[\left|u_{k+1}-u_{k}\right|=\mathcal{O}(h^{p+1})\quad
\text{and}\quad\left|U_{kj}-u({t_k+c_jh})\right|=\mathcal{O}(h^{q+1}),\quad i=1,\ldots,r.\]
Let $I$ be the identity matrix in $\R^{r\times r}$ and $\mathds{1}=(1,\ldots,1)^\top$, 
the corresponding stability function is given by
\begin{equation}\label{StabFunc}
R(z)=1+z \mathbf{b} ^\top (I-z  A )^{-1}\mathds{1}.
\end{equation}

\begin{assumption}\label{assumption}
From now on, we make the following assumption:
\begin{enumerate}
\item The matrix $A$ is invertible.
\item  $R(iy)\leq 1$, for all real $y$.
\item $R(z)$ is analytic for $\mathrm{Re}(z)<1$.
\end{enumerate}
\end{assumption}

The two last assumptions guarantee $A$-stability of the method, in other words, the method is stable on the entire negative half-complex plane $\mathbb{C}^-$.
Important examples of RK methods satisfying these assumptions are Radau IIA (with order $p = 2r - 1$ and stage order $q = r$) and Lobatto IIIC methods (with order $p = 2r - 2$ and stage order $q = r-1$).

\subsubsection{Runge-Kutta Convolution Quadratures}
Based on such RK methods, one can construct specific convolution quadratures \cite{LuOs, BaLu,BaLo}.  To this end, we define time grids $\tau:=\{t_k+c_ih\}_{k=0, \ldots,N}^{i=1, \ldots, r}$ and $\mathbf{t}_k:=\{t_k+c_ih\}_{i=1, \ldots,r}$ with $t_k=kh$, $t_k^i=t_k+c_ih$ and let $f:[0,T]\Flder\R^d$. Consider the collection of vectors $f(\mathbf{t}_k)=\mathbf{f}_k=\{f_k^i=f(t_k^i)\}_{i=1,\ldots,r}=\left(f(t_k^1),\ldots,f(t_k^r)\right)^\top\in\mathbb{R}^{r\times d}$ for $k=0,\ldots,N$. Then, using the \textit{retarded} RKCQ methods, the approximation of the fractional integral $J_-^\alpha f(\mathbf{t}_k)$   is given by
\begin{equation}\label{RKConQua}
J_{-}^{\alpha}f(\mathbf t_k)\approx\mathcal{J}_{-}^{\alpha}f(\mathbf t_k):=\sum_{n=0}^{k}W_{k-n}^{(\alpha)}{\bf f}_{n}=\sum_{n=0}^{k}W_{k-n}^{(\alpha)}f({\bf t}_n)
\end{equation}
or explicitly
$$\mathcal{J}_{-}^{\alpha}{\bf f}_k= \sum_{n=0}^kW_{k-n}^{(\alpha)}\begin{pmatrix}
{f}^{1}_{n}\\
\vdots\\
{f}^{ r }_{n}
\end{pmatrix}.$$
Here, $K^{(\alpha)}$ is again as defined \eqref{LaplaceTrans} and the matrices $W_n^{(\alpha)}\in\R^{ r \times  r}$ are the  {\it convolution weights},  determined by the series expansion
\[
K^{(\alpha)}\lp\frac{\gamma(z)}{h}\rp=\sum_{n=0}^{\infty}W_n^{(\alpha)}z^n,
\]
where the matrix-valued function $\gamma(z)$, based on the underlying RK method, is given by 
\[
\gamma(z)=\lp  A +\frac{z}{1-z}\mathds{1}\,\mathbf b^\top \rp^{-1}= A ^{-1}-z A ^{-1}\mathds{1}\mathbf  b^\top  A ^{-1}.
\]
Using a different notation, which will be convenient later, the approximation at the stage time $t_k^i$ is defined by the $i$-th row of \eqref{RKConQua} and is be computed   as \footnote{Again, the convolution quadrature can equivalently  be defined as
\[
\mathcal{J}_{-}^{\alpha}\mathbf f_k=\sum_{n=0}^{k}W_{n}^{(\alpha)}f(\mathbf t_{k-n})
\]
}
\[
\mathcal{J}_{-}^{\alpha}f_k^i=\sum_{n=0}^{k} \sum_{j=1}^{r} \left[W_n^{(\alpha)}\right]_{j}^if_{k-n}^j, \quad i=1,\ldots, r.
\]
where $\left[W_n^{(\alpha)}\right]_{j}^i$ denotes the 
$(i,j)$-th entry of the weight matrix $W_n^{(\alpha)}$.
An efficient way to  compute the weights can be achieved by using fast Fourier transformation (FFT) as described in \cite{LuOs}. Following \cite[Section 5.4]{LehSaBook}, the convolution weights $W_n^{(\alpha)}$ can be represented as a Cauchy integral formula as follows

\begin{align}
    W_n^{(\alpha)} &=\frac{1}{2\pi i}\int_{|z|=\lambda}K^{(\alpha)}\left(\frac{\gamma(z)}{h}\right)z ^{-n-1} \dd z\nonumber\\
 &=\lambda^{-n}\int_{0}^1 K^{(\alpha)}\left(\frac{\gamma(\lambda e^{-2\pi i \theta})}{h}\right)e ^{2\pi i \theta n} \dd \theta,\quad 0<\lambda<1.\nonumber\\
 &\approx \frac{\lambda^{-n}}{N+1}\sum_{\ell=0}^N K^{(\alpha)}\left(\frac{\gamma(z_\ell)}{h}\right)z_\ell^{-n}
\quad \text{(trapezoidal rule)}\label{TR}
\end{align}
where $z_\ell=\lambda e^{-2\pi i\ell/(N+1)}$. The  following  algorithm  summarizes  the  basic  steps  for  computing the convolution weights
\begin{algorithm}{}
\begin{algorithmic}[1]
\State {\bf Initial data:} $N$, $h$, $\gamma(z)$, $K^{(\alpha)}(s)$ and quadrature points $z_\ell,\, \ell =0,\ldots,N$
 \For {$\ell= 0,\ldots, N$} 
\State {\bf Diagonalize} $\gamma(z_\ell)=V_\ell D_\ell V_\ell^{-1}$
\State {\bf Apply} $K^{(\alpha)}$: $K^{(\alpha)}\left(\frac{\gamma(z_\ell)}{h}\right)=V_\ell K^{(\alpha)}(D_\ell/h) V_\ell^{-1}$,
\State {\bf Use} the FFT in equation \eqref{TR}
\EndFor
\State {\bf Output:} $W_0^{(\alpha)},W_1^{(\alpha)},\ldots,  W_N^{(\alpha)}.$
\end{algorithmic}
\caption{Basic algorithm to compute the convolution weights}
\end{algorithm}

A good choice of $\lambda$ is $\lambda=\varepsilon^{1/(2N+1)}$ yields an error in $W_n^{(\alpha)}$ of size $\sqrt{\varepsilon}$; see \cite{Lubich2}. Another representation of the weights $W_n^{(\alpha)}$  can be found in  \cite{BaLo}, leading also to an efficient algorithm approximating fractional integrals.

Now, we state the  convergence order of the RKCQ \eqref{RKConQua} as shown in \cite[Theorem 5]{BaLuMar}. Here, we omit the regularity conditions on the general Laplace transform $K(s)$ for which the original theorem is established, since our focus is on $K^{(\alpha)}(s)$ \eqref{LaplaceTrans}, which alraedy satisfies these conditions.  Moreover, under suitable assumptions on the function $f$ and its derivatives, the convolution quadrature converges with order $\min( p, q + 1 - \mu )$.

\begin{theorem}\label{thm:RKCQ_convergence}
Let $K^{(\alpha)}$ be the Laplace transform  of the kernel which  is analytic in the half-plane $\mathrm{Re} s > \sigma > 0$, such that for
some real exponent $\mu$ and bounding factor $C_K(\sigma)> 0$, the operator norm is bounded as follows:
\begin{equation*}
    \left\|K^{(\alpha)}(s)\right\|\leq C_K(\sigma) |s|^\mu\quad \text{for all}\quad \mathrm{Re} s>\sigma.
\end{equation*}
Let a Runge-Kutta method which satisfies Assumption \ref{assumption}, $r> \max(p+\mu+1,p,q+1)$ and  $f\in C^r([0,T],\R)$, then there exist $\bar h>0$ and $C>0$ such that for $h\leq \bar h$ it holds
\begin{equation}
\left\|J^{\alpha}_{-}f(t_k^i)-\mathcal{J}_{-}^{\alpha}f(t_k^i)\right\|\leq C h^{\min(p,q+1-\mu)}\left(\left|f^{(r)}(0)\right|+\int_0^t\left|f^{(r+1)}(s)\right|\dd s\right),
\end{equation}
where $C$ depends on the constants $\bar h,T$ and the underlying RK method, but is independent of $h$ and $f$.
\end{theorem}
The previous theorem is applicable  to the fractional integral \eqref{RLInt:a} when $\mu=\alpha<0$ and to the fractional derivative \eqref{RLDer:a} in the case $\mu>0$.  The convolution is then defined by
\begin{equation}\label{RKConQua2}
D_{-}^{\alpha}f(\mathbf t_k)\approx\mathcal{D}_{-}^{\alpha}\mathbf f_k:=\sum_{n=0}^{k}W_{k-n}^{(-\alpha)}{\bf f}_{n}.
\end{equation}
An important observation is that Theorem \ref{thm:RKCQ_convergence} indicates that stage order limits the convergence for $\alpha>0$. Consequently, Runge–Kutta methods with low stage order cannot guarantee full accuracy for higher-order fractional derivatives.

\section{Fractional Variational Integrators with RKCQ}\label{sec:FVI}
In this section we derive fractional integrators for restricted variational principle using  Runge-Kutta convolution quadratures. We first recall the restricted variational principle as presented in \cite{JiOb1,JiOb2}, which forms the foundation for the construction of such integrators.

\subsection{Continuous Setting of Restricted Variational Principle}\label{sec:restrectedCOV}
Let  $x,y\in AC^2([0,T],\mathbb R^d)$ and  $L:\mathsf{T}\R^d\Flder \R$ be a $C^2$-Lagrangian (using the natural identification $\mathsf{T}\R^d\cong\R^d\times\R^d$). Define the augmented fractional Lagrangian
\begin{equation}\label{GenerLag}
\begin{array}{rcl} L^{(\alpha)}:\R^d\times\R^d\times\R^d\times\R^d\times\R^d\times\R^d &  \Flder  & \R \\
 (x,\,y,\,\dot x,\,\dot y,\, D^{\alpha}_{-}x,\, D^{\alpha}_{+}y)    & \mapsto&L^{(\alpha)}(x,y,\dot x,\dot y, D^{\alpha}_{-}x,D^{\alpha}_{+}y)\\ [2ex]
  L^{(\alpha)}(x,\,y,\,\dot x,\,\dot y,\, D^{\alpha}_{-}x,\, D^{\alpha}_{+}y)=L(x,\dot x)+&&\hspace{-1cm}L(y,\dot y)-\rho\,D^{\alpha}_{-}x\,D^{\alpha}_{+}y,
\end{array}
\end{equation}
where $\rho >0$. With this Lagrangian, we define the relevant action $\mathcal{L}: AC^2([0,T])\times AC^2([0,T])\to \R$ by  $\mathcal{L}(x,y)=\Lc(x,y)+\Lf(x,y)$ with
\begin{equation}\label{FracAction}
\begin{split}
\Lc(x,y)&=\int_0^T(L(x(t),\dot x(t))+L(y(t),\dot y(t))\,)\, \dd t\\
\Lf(x,y)&=-\rho\,\int_0^T D^{\alpha}_{-}x(t)\,D^{\alpha}_{+}y(t)\,\dd t.
\end{split}
\end{equation}
Moreover,  we define restricted varied curves (both curves have the same variation) as
\begin{equation}\label{VariedCurves}
x_{\epsilon}(t)=x(t)+\epsilon\, \delta x(t),\,\,\,y_{\epsilon}(t)=y(t)+\epsilon\, \delta x(t),\quad \epsilon>0
\end{equation}
and a $AC^2$-curve $\delta x:[0,T]\Flder\R^d$ such that $\delta x(0)=\delta x(T)=0$. Now, assuming fixed endpoints $x(0),\,x(T),\,y(0),\,y(T)$, we have the restricted fractional Euler-Lagrange equations:
\begin{theorem}[FEL equations]\label{ContTheo}
The equations
\begin{subequations}\label{ContFracDamp}
\begin{align}
\frac{\dd}{\dd t}\lp\frac{\der L(x,\dot x)}{\der\dot x}\rp-\frac{\der L(x,\dot x)}{\der x}=-\rho\,D^{(2\alpha)}_{-}x,\label{ContFracDamp:a}\\
\frac{\dd}{\dd t}\lp\frac{\der L(y,\dot y)}{\der\dot y}\rp-\frac{\der L(y,\dot y)}{\der y}=-\rho\,D^{(2\alpha)}_{+}y,\label{ContFracDamp:b}
\end{align}
\end{subequations}
are  {\rm sufficient} conditions for the extremals of $\L(x,y)$ \eqref{FracAction} under restricted variations \eqref{VariedCurves}.
\end{theorem}

\subsection{Discrete Setting Based on RKCQ}

We analogously define the \textit{advanced} RKCQ by 
\begin{equation}\label{RKConQuaPlus}
\mathcal{J}_{+}^{\alpha}f(\mathbf t_k)\approx\mathcal{J}_{+}^{\alpha}f(\mathbf t_k):=\sum_{n=0}^{N-k}\left[W_{n}^{(\alpha)}\right]^\top\mathbf{f}_{k+n}.
\end{equation}
 where the transpose relationship  in the equality can be followed from the definition of the series, namely
\[
\left[K^{(\alpha)}\lp\frac{\gamma(z)}{h}\rp\right]^\top=\left[\sum_{n=0}^{\infty} W_n^{(\alpha)}z^n\right]^\top = \sum_{n=0}^{\infty} \left[W_n^{(\alpha)}\right]^\top z^n.
\]

\begin{remark}
In the formulation of the discrete restricted Hamilton's principle, all variables must be treated consistently within the discrete framework. Therefore, the function 
$f$ should be represented by its discrete approximation 
$f_d=\{\mathbf f_k\}_{0:N}$, defined at the discrete time. Accordingly, the convolution quadratures \eqref{RKConQua} and \eqref{RKConQuaPlus}  should be constructed using these discrete values, rather than than using the evaluation of the continuous function $f(t)$ at grid time $t_k^i$. This ensures consistency with the fully discrete variational framework and avoids mixing continuous and discrete representations.
\end{remark}

With these ingredients, we can establish the asymmetric integration by parts for the RKCQ which we need to derive the restricted discrete dynamics.

\begin{lemma}\label{ConvPropertiesRK1}
Consider two discrete series $\left\{ {\bf f}_k\right\}_{0:N}, \left\{ {\bf g}_k\right\}_{0:N}$, where ${\bf f}_k=(f_k^1,.\ldots,f_k^r)^\top\in \R^r$ and ${\bf g}_k=(g_k^1,\ldots,g_k^r)^\top\in \R^r$. Then the following properties hold true:
\begin{enumerate}
    \item The semigroup property of the discrete convolution: $\mathcal{J}_{\sigma}^{\alpha}\circ \mathcal{J}_{\sigma}^{\alpha}{\bf f }_k=\mathcal{J}_{\sigma}^{(2\alpha)}{\bf f }_k$, $\sigma\in \{-,+\}$.
\item  The asymmetric integration by parts:
 \begin{equation}\label{AsymmetricInt:2}
\sum_{k=0}^{N}\left\langle {\bf g}_k, \mathcal{J}_{-}^{\alpha}{\bf f }_k\right\rangle=\sum_{k=0}^{N}\left\langle \mathcal{J}_{+}^{\alpha}{\bf g}_k, {\bf f }_k\right\rangle.
\end{equation}  
\end{enumerate}
\end{lemma}
\begin{proof}
The proof of the fist property is analogous to that of \cite[Lemma 4.1]{HaJiOb}. For the second one, using the definition \eqref{RKConQua} and employing lower index for the row vectors yield the following
\[
\sum_{k=0}^{N}\left\langle {\bf g}_k, \mathcal{J}_{-}^{\alpha}{\bf f }_k\right\rangle=\sum_{k=0}^{N}\sum_{n=0}^{k}{\bf g}_k^\top W_n^{(\alpha)}\,{\bf f}_{k-n}=\sum_{k=0}^{N}\sum_{n=0}^{k}\sum_{i,j=1}^r(g_k)_i\left[W_n^{(\alpha)}\right]_{j}^i f_{k-n}^j;
\]
From this last expression, extracting the $i,j$ sums to the front, we have 
\[
\begin{split}
&\sum_{i,j=1}^r\sum_{k=0}^{N}\sum_{n=0}^{k}(g_k)_i\left[W_n^{(\alpha)}\right]_{j}^if_{k-n}^j=^{1}\sum_{i,j=1}^r\sum_{n=0}^N\sum_{k=n}^N(g_k)_i\left[W_n^{(\alpha)}\right]_{j}^i f_{k-n}^j\\
&=^{2}\sum_{i,j=1}^r\sum_{n=0}^N\sum_{k=0}^{N-n}(g_{k+n})_i\left[W_n^{(\alpha)}\right]_{j}^if_{k}^j=^{3}\sum_{i,j=1}^{r}\sum_{k=0}^N\sum_{n=0}^{N-k}(g_{k+n})_i \left[W_n^{(\alpha)}\right]_{j}^if_{k}^j\\
&=^4\sum_{k=0}^N\sum_{n=0}^{N-k}\sum_{i,j=1}^r(g_{k+n})_i\left[W_n^{(\alpha)}\right]_{j}^if_{k}^j=^5\sum_{k=0}^N\sum_{n=0}^{N-k} \mathbf g_{k+n} W_n^{(\alpha)}\mathbf f_{k}\\
&=^6\sum_{k=0}^N(\mathcal{J}_{+}^{\alpha}{\bf g}_k)^\top\, {\bf f}_k.
\end{split}
\]
Observe that in $=^{1,2,3}$ we have employed the same relationships as in Lemma \ref{ConvPropertiesRK1}, since for these, which only involve the index $k$,  the indices $i,j$ are not affected. In $=^4$ we have rearranged the sum order, in $=^5$ we have employed matrix notation and finally, in $=^6$, we use the definition \eqref{RKConQuaPlus}; and the result follows.
\end{proof}

Before deriving fractional variational integrators with RKCQ \S\ref{sec:FVI}, which are fully characterized by Butcher Tableau, we first establish the following two technical lemmas.  
\begin{lemma}\label{ConvPropertiesRK2}
Consider two discrete series $\left\{ {\bf f}_k\right\}_{ 0:N}, \left\{ {\bf g}_k\right\}_{0:N}$, where ${\bf f}_k=(f_k^1,.\ldots,f_k^r)^\top\in \R^r$ and ${\bf g}_k=(g_k^1,\ldots,g_k^r)^\top\in \R^r$ and let $\mathbf{B}=\diag(b),\ b\in \mathbb R^r$ Then the following property holds true:
\begin{equation}\label{AsymmetricInt:3}
\sum_{k=0}^{N}\left\langle {\bf g}_k, \mathcal{J}_{-}^{\alpha}{\bf B f }_k\right\rangle=\sum_{k=0}^{N}\left\langle {\bf B}\mathcal{J}_{+}^{\alpha}{\bf g}_k, {\bf f }_k\right\rangle.
\end{equation}
\end{lemma}

\begin{proof}
The proof is analogous to that of Lemma \ref{ConvPropertiesRK1}.
\end{proof}

 In what follows, we set $r=s+1$. Moreover, we focus on Lobatto IIIC methods for which $c_1=0$ and $c_{r}=1$ and we denote $W_k^{(\alpha)}$ the matrix weights associated to the fractional derivatives, instead of $W_k^{(-\alpha)}$.
\begin{table}[H]
\centering
$\setlength{\tabcolsep}{20pt}
\renewcommand{\arraystretch}{1.5}\begin{array}{c|cc}
0  &  \frac{1}{2} & - \frac{1}{2} \\
   1  & \frac{1}{2} &\phantom{-}\frac{1}{2} \\ \hline 
   & \frac{1}{2}& \frac{1}{2}
\end{array}\hspace{2cm}\begin{array}{c|ccc}
0  &  \frac{1}{6} & - \frac{1}{3} & \phantom{-}\frac{1}{6} \\
   \frac{1}{2}  &\frac{1}{6} &  \phantom{-}\frac{5}{12} & -\frac{1}{12} \\
 1   &\frac{1}{6} & \phantom{-} \frac{2}{3} & \phantom{-}\frac{1}{6}  \\ \hline 
 &\frac{1}{6} &  \phantom{-}\frac{2}{3} & \phantom{-}\frac{1}{6} 
\end{array}\hspace{2cm}\begin{array}{c|cccc}
0  &  \frac{1}{12} & - \frac{\sqrt{5}}{12} & \frac{\sqrt{5}}{12}&-\frac{1}{12} \\
\frac{5-\sqrt{5}}{10} &  \frac{1}{12} &  \frac{1}{4} & \frac{10-7\sqrt{5}}{60}&\phantom{-}\frac{\sqrt{5}}{60} \\
\frac{5+\sqrt{5}}{10}  &  \frac{1}{12} & \frac{10+7\sqrt{5}}{60} & \frac{1}{4}&-\frac{\sqrt{5}}{60} \\
1  &  \frac{1}{12} &  \frac{5}{12} & \frac{5}{12}&\phantom{-}\frac{1}{12} \\ \hline
 &\frac{1}{12} &  \frac{5}{12} & \frac{5}{12} & \phantom{-}\frac{1}{12}
\end{array}$
\caption{Lobatto IIIC methods with $r = 2$, $r = 3$  and $r = 4.$}
\label{tab:Lobatto}
\end{table}

We consider two discrete series $x_d:=\{\mathbf{x}_k\}_{0:N}$ and  $y_d:=\{\mathbf{y}_k\}_{0:N}$ with $\mathbf{x}_k:=(x_k^1,\ldots,x_k^{s+1})^\top$ and $\mathbf{y}_k:=(y_k^1,\ldots,y_k^{s+1})^\top$ such that $(x_{k-1}^{s+1},y_{k-1}^{s+1})=(x_{k}^{1},y_{k}^{1})$. 
Regarding the approximation of \eqref{FracAction}, we consider the following discrete action:  
\begin{equation}\label{DiscFracActionHORK}
\begin{split}
\mathcal{L}_d(z_d)=\Lc_d&(x_d,y_d)+\Lf_d(x_d,y_d), \\
\Lc_d(x_d,y_d)=\sum_{k=0}^{N-1}\big[L_d(x_k)+L_d(y_k)\big],&\quad \Lf_d(x_d,y_d)=-\rho\,h\,\,\sum_{k=0}^{N}\left\langle\widetilde{\mathcal{D}}^{\alpha}_{+}{\bf y}_k,\mathcal{D}^{\alpha}_{-}{\bf x}_k\right\rangle,\\
L_d(x_k)=h\sum_{i=1}^rb_iL(x_d(c_i\,h;k),\,\dot x_d(c_i\,h;k)),&\,\,\,\,\,\, L_d(y_k)=h\sum_{i=1}^rb_iL(y_d(c_i\,h;k),\dot y_d(c_i\,h;k)),
\end{split}
\end{equation}
where we employ the RK convolution quadratures \eqref{RKConQua}, \eqref{RKConQuaPlus} with $\widetilde{\mathcal{D}}^{\alpha}_{+}{\bf y}_k:=\mathcal{D}^{\alpha}_{+}{\bf B y}_k$. Now, we again consider  restricted varied curves in order to establish the variational principle. In vectorial form, these  varied curves read read as follows
\begin{equation}\label{RestrictedVariedInnerRK}
x_d^{\epsilon}=\left\{\mathbf{x}_k+\epsilon \delta \mathbf{x}_k\right\}_{0:N},\qquad y_d^{\epsilon}=\left\{\mathbf{y}_k+\epsilon \delta \mathbf{x}_k\right\}_{0:N}
\end{equation}
where $\delta \mathbf x_k=(\delta x_k^1,\ldots,\delta x_k^{s+1})^\top$. Naturally, we will set 
$x_N^{i}=0$ for $i=2,\ldots,s+1$ and  $\delta\mathbf  x_N=0$ (equivalent for $y$), besides $\delta x_{0}^{1}=\delta y_{0}^{1}=\delta x_{N-1}^{s+1}(=\delta x_{N}^{1})=0=\delta y_{N-1}^{s+1}(=\delta y_{N}^{1})$, since $x_{0}^{1},\,y_{0}^{1},\,x_{N-1}^{s+1}=x_{N}^{1},\,y_{N-1}^{s+1}=y_{N}^{1}$  are fixed.

The following lemma is needed for the proof of the main theorem. In its proof and in the rest of the paper, we adopt the notation $[A]^i$ for the $i$-th row of a matrix $A$.
\begin{lemma}\label{VariationCommRK}
According to the definitions \eqref{RKConQua}, \eqref{RKConQuaPlus} and considering varied curves \eqref{RestrictedVariedInnerRK}, we have that 
\[
\delta \mathcal{D}^{\alpha}_{-}{\bf x}_k = \mathcal{D}^{\alpha}_{-}{\bf \delta x}_k\qquad\text{and}\qquad \delta \mathcal{D}^{\alpha}_{+}{\bf y}_k = \mathcal{D}^{\alpha}_{+}{\bf \delta y}_k.
\]
\end{lemma}

\begin{proof}
We pick $\delta \mathcal{D}^{\alpha}_{-}{\bf x}_k$; the proof for $\delta \mathcal{D}^{\alpha}_{+}{\bf y}_k$ is similar. We have for  a fixed $i$
\begin{align*}
\delta \mathcal{D}^{\alpha}_{-}{\bf x}_k=\frac{\dd}{\dd\epsilon}\bigg|_{\epsilon=0}\mathcal{D}^{\alpha}_{-}{\bf x}_d^{\epsilon}
&=\frac{\dd}{\dd\epsilon}\bigg|_{\epsilon=0} \sum_{n=0}^{k}\sum_{j=1}^{r}\left[W_n^{(\alpha)}\right]^i_{j}(x_{k-n}^j+\epsilon\,\delta x_{k-n}^j)\\
&=\sum_{n=0}^{k}\sum_{j=1}^{r}\left[ W_n^{(\alpha)}\right]^i_{j}\delta x_{k-n}^j=\mathcal{D}^{\alpha}_{-}\,{\bf \delta x}_k.
\end{align*}
\end{proof}

\begin{theorem}[FDEL equations]\label{thm:maintheorem}
Let us consider two discrete series $\left\{{\bf x}_k\right\}_{0:N}$ and $\left\{{\bf y}_k\right\}_{0:N}$ such that $x_N^{2:s+1}=y_N^{2:s+1}=0$. Then, the equations
\begin{subequations}\label{FinalRK}
\begin{align}
&D_{s+1}L_d(x_{k-1}^{1},\ldots,x_{k-1}^{s+1})+ D_1L_d(x_k^1,\ldots,x_k^{s+1})\nonumber\\
&\hspace{3cm}-\rho h\left(b_1\left[\mathcal{D}^{(2\alpha)}_{-}{\bf x}_k\right]^1 +b_s\left[ \mathcal{D}^{(2\alpha)}_{-}{\bf x}_{k-1}\right]^{s+1}\right)=0,\quad  k=1,\ldots, N-1,\label{FinalRK1}\\
&D_iL_d(x_k^1,\ldots,x_k^{s+1})-\rho h b_i \left[\mathcal{D}^{(2\alpha)}_{-}{\bf x}_k\right]^i=0,\,\hspace{1cm}\,\,\quad k=0,\ldots, N-1,\,\quad i=2,\ldots,s,\label{FinalRKi}\\
&D_{s+1}L_d(y_{k-1}^{1},\ldots,y_{k-1}^{s+1})+ D_1L_d(y_k^0,\ldots,y_k^s)\nonumber\\
&\hspace{3cm}-\rho h\left(\left[\mathcal{D}^{(2\alpha)}_{+}{\bf B y}_k\right]^1 +\left[\mathcal{D}^{(2\alpha)}_{+}{\bf B y}_{k-1}\right]^{s+1}\right)=0,\quad  k=1,\ldots, N-1,\\
&D_iL_d(y_k^{1},\ldots,y_k^{s+1})-\rho h \left[\mathcal{D}^{(2\alpha)}_{+}{\bf B y}_k\right]^i=0,\,\hspace{1cm}\,\,\quad k=0,\ldots, N-1,\,\quad i=2,\ldots,s,
\end{align}
\end{subequations}
 are a sufficient condition for the extremals of \eqref{DiscFracActionHORK} under restricted calculus of variations \eqref{RestrictedVariedInnerRK}.
\end{theorem}

\begin{proof}
For the conservative part, we have under the restricted variation $\delta \mathbf x=  \delta\mathbf  y$ that
\begin{equation}\label{eq:variation_proof}
\delta\Lc_d(x_d,y_d)=\sum_{k=0}^{N-1}\sum_{i=1}^{s+1}\lp D_iL_d(x_k)+D_iL_d(y_k)\rp\,\delta x_k^{i},
\end{equation}
where $D_iL_d(x_k^{1},\ldots,x_k^{s+1})\delta x_k^{i}:=\frac{\partial L_d(x_k)}{\partial x_k^{i}}\delta x_k^{i}$. For the fractional part, we employ the Leibnitz rule of the derivative and Lemma \ref{VariationCommRK}, we obtain
\[
\delta\Lf_d(x_d,y_d)=-\rho\,h\sum_{k=0}^{N}\left\langle\widetilde{\mathcal{D}}^{\alpha}_{+}{\bf \delta\, y}_k,\mathcal{D}^{\alpha}_{-}{\bf x}_k\right\rangle -\rho\,h\sum_{k=0}^{N}\left\langle\widetilde{\mathcal{D}}^{\alpha}_{+}{\bf y}_k,\mathcal{D}^{\alpha}_{-}{\bf \delta\, x}_k\right\rangle.
\]
Using Lemmas \ref{ConvPropertiesRK1} and \ref{ConvPropertiesRK2}, and taking into account that $\delta {\bf x}_k = \delta {\bf y}_k$, we have
\[
\begin{split}
\delta\Lf(x_d,y_d)=^1&-\rho\,h\sum_{k=0}^{N}\left\langle{\bf \delta\, x}_k,\mathbf{B}\mathcal{D}^{(2\alpha)}_{-}{\bf x}_k\right\rangle -\rho\,h\sum_{k=0}^{N}\left\langle\mathcal{D}^{(2\alpha)}_{+}{\bf B y}_k,{\bf \delta\, x}_k\right\rangle\\
=^2&-\rho\,h\sum_{k=0}^{N}\delta {\bf x}_k^\top (\mathbf{B}\mathcal{D}^{(2\alpha)}_{-}{\bf x}_k) -\rho\,h\sum_{k=0}^{N}(\mathcal{D}^{(2\alpha)}_{+}{\bf B y}_k)^\top\,\delta {\bf x}_k\\
=^3&-\rho\,h\sum_{k=0}^{N} \delta x_k^{\nu_i}\left[\mathbf{B}\mathcal{D}^{(2\alpha)}_{-}{\bf x}_k\right]^i -\rho\,h\sum_{k=0}^{N}\left[\mathcal{D}^{(2\alpha)}_{+}{\bf B y}_k\right]^i\delta  x_k^{\nu_i}\\
=^4&-\rho\,h\sum_{k=0}^{N} \lp\delta_{ij}\left[\mathbf{B}\mathcal{D}^{(2\alpha)}_{-}{\bf x}_k\right]^j +\left[\mathcal{D}^{(2\alpha)}_{+}{\bf B y}_k\right]^i\rp\,\delta  x_k^{\nu_i}\\
=^5&-\rho\,h\sum_{k=0}^{N-1} \lp\delta_{ij}\left[\mathbf{B}\mathcal{D}^{(2\alpha)}_{-}{\bf x}_k\right]^j +\left[\mathcal{D}^{(2\alpha)}_{+}{\bf B y}_k\right]^i\rp\,\delta  x_k^{\nu_i};
\end{split}
\]
in $=^1$ we have also employed the semigroup property and expanded the sum over $i$; in $=^4$ we have used the Euclidean metric $\delta_{ij}$ and rearranged terms; finally, in $=^5$ we have $\delta {\bf x}_N=0$.  Putting everything together we have
\[
\begin{split}
\delta\mathcal{L}_d(x_d,y_d)=&\sum_{k=0}^{N-1}\sum_{i=1}^{s+1}\lp D_iL_d(x_k)+D_iL_d(y_k)-h\rho\,\delta_{ij}(\mathbf{B}\mathcal{D}^{(2\alpha)}_{-}{\bf x}_k)^j -h\rho\,\left[\mathcal{D}^{(2\alpha)}_{+}{\bf By}_k\right]_i\rp\delta x_k^{i}\\
=&\sum_{k=0}^{N-1}\sum_{i=1}^{s+1}\lp D_iL_d(x_k)-h\rho\,\delta_{ij}\left[\mathbf{B}\mathcal{D}^{(2\alpha)}_{-}{\bf x}_k\right]^j\rp\delta x_k^{i}\\
&\qquad\qquad\qquad+\sum_{k=0}^{N-1}\sum_{i=1}^{s+1}\lp D_iL_d(y_k)-h\rho\,\left[\mathcal{D}^{(2\alpha)}_{+}{\bf B y}_k\right]_i\rp\delta x_k^{i}\\
=&\sum_{k=0}^{N-1}\sum_{i=2}^{s}\lp D_iL_d(x_k)-h\rho\,\delta_{ij}\,\left[\mathbf{B}\mathcal{D}^{(2\alpha)}_{-}{\bf x}_k\right]^j\rp\delta x_k^{i}\\
&\quad\quad\quad+\sum_{k=1}^{N-1}\lp D_{s+1}L_d(x_{k-1})+ D_1L_d(x_k)-h\rho\,\delta_{1j}\,\left[\mathbf{B}\mathcal{D}^{(2\alpha)}_{-}{\bf x}_k\right]^j\rp\delta x_k^{1}\\
+&\sum_{k=0}^{N-1}\sum_{i=2}^{s}\lp D_iL_d(y_k)-h\rho\,\left[\mathcal{D}^{(2\alpha)}_{+}{\bf B y}_k\right]_i\rp\delta x_k^{i}\\
&\quad\quad\quad+\sum_{k=1}^{N-1}\lp D_{s+1}L_d(y_{k-1})+ D_1L_d(y_k)-h\rho\,\left[\mathcal{D}^{(2\alpha)}_{+}{\bf By}_k\right]_1\rp\delta x_k^{1},
\end{split}\]
where it is taken into account that $\delta x_0^1=0$ and that $D_{s+1}L_d(x_{k-1})=D_1L_d(x_k).$ Now, given that $\delta x_k^{i}$ are arbitrary for $k=0,\ldots,N-1$, $i=1,\ldots,s+1$ (except $\delta x_0$), we obtain the system \eqref{FinalRK} from the last equality and the theorem follows.
\end{proof}

The FDEL equations \eqref{FinalRK1} and \eqref{FinalRKi}, corresponding to the $x$-part, define a discrete iteration scheme for the fractional dynamics \eqref{ContFracDamp:a}. The general FDEL scheme with RKCQ is summarized in Algorithm \ref{alg:FractionalAlgorithm}. As the system is nonlinear, each step is typically solved using a numerical root-finding, such as a Newton method, to compute the $s$ unknowns $(x_k^1,\ldots,x_k^{s+1}=x_{k+1}^1)$. Moreover, since fractional derivatives are non-local, the implementation must account for the full past of the function, which typically requires additional interpolation or extrapolation rather than using only local values. In Algorithm \ref{alg:FractionalAlgorithm}, this extrapolation is performed for the term $\mathcal{D}^{(2\alpha)}_{-}{\bf x}_{k-1}$.

\begin{algorithm}{}
\begin{algorithmic}[1]
\State {\bf Initial data}: $N,\, h,\,\alpha,\,W_n^{(\alpha)},\, x_0^1,\, p_{x_0}.$
 \State {\bf Solve for} $x_0^2, \ldots,x_0^{s+1}$ {\bf from} 

\[
\begin{split}
p_{x_0}=&-D_1L_d(x_0^1, \ldots,x_0^{s+1})+\rho hb_1\left[\mathcal{D}^{(2\alpha)}_{-}{\bf x}_k\right]^1,\\
0=&\phantom{-} D_iL_d(x_0^1, \ldots.,x_0^{s+1})-\rho hb_i\left[\mathcal{D}^{(2\alpha)}_{-}{\bf x}_k\right]^i,\quad \forall\,i=2, \ldots,s.
\end{split}
\]
\State {\bf Initial points:} $x_0^1, \ldots,x_0^{s+1}=x_1^1$
    \For {$k= 1: N-1$} 
    
\State  {\bf Solve for} $x_{k}^2, \ldots,x_k^{s+1}=x_{k+1}^1$ {\bf from} 
\[
\begin{split}
0=&\,\,D_{s+1}L_d(x_{k-1}^1, \ldots,x_{k-1}^{s+1})+ D_1L_d(x_k^1, \ldots,x_k^{s+1})-\rho h \left(b_1\left[\mathcal{D}^{(2\alpha)}_{-}{\bf x}_k\right]^1+b_{s+1}\left[\mathcal{D}^{(2\alpha)}_{-}{\bf x}_{k-1}\right]^{s+1}\right),\\
0=&\,\,D_iL_d(x_k^1, \ldots,x_k^{s+1})-\rho h b_i\left[\mathcal{D}^{(2\alpha)}_{-}{\bf x}_k\right]^i,\quad \forall\,i=2, \ldots,s.
\end{split}
\]
    \EndFor
    \State  {\bf Output:} $(x_1^{i}, \ldots,x_{N-1}^{i}),\quad i=1, \ldots,s+1.$
\end{algorithmic}
\caption{FVI with RKCQ}\label{alg:FractionalAlgorithm}
\end{algorithm}

\subsection{Approximation Order} The restricted  Hamiltonian system associated to the FEL equation \eqref{ContFracDamp:a} is defined as
\begin{align}\label{eq:xpart_frac_hamiltonian_sys}
\begin{split}
\begin{dcases}
\dot x &= \phantom{-}p\\
\dot 	p &= -\nabla U(x)-\rho  D_-^{(2\alpha)}x. 
\end{dcases}
\end{split}
\end{align}
The error analysis  for FVI \eqref{FinalRK1}-\eqref{FinalRKi} requires the definition of discrete momenta, and it can be done by applying the discrete Legendre transform. In this section, we focus on the  $2$-stage Lobatto IIIC method.
\begin{align}
        p_k^-&=-D_1L_d(x_k,x_{k+1})+\frac{\rho h}{2} \left[\mathcal{D}^{(2\alpha)}_{-}{\bf x}_k\right]^1\\
   p_{k+1}^+&=\phantom{-}D_2L_d(x_k,x_{k+1})-\frac{\rho h}{2} \left[\mathcal{D}^{(2\alpha)}_{-}{\bf x}_k\right]^2.
\end{align}
For the mechanical Lagrangian $L(q,\dot q)=\dot q^\top \dot q -U(q)$, the associated approximation to the $2$-stage Lobatto IIIC method reads
\begin{align}
    L_d(x_k,x_{k+1})&=\frac{h}{2}L\left(x_k,\frac{x_{k+1}-x_k}{h}\right)+\frac{h}{2}L\left(x_{k+1},\frac{x_{k+1}-x_k}{h}\right)\\
&=\frac{1}{2h}\left(x_{k+1}-x_k\right)^\top \left(x_{k+1}-x_k\right)-\frac{h}{2}\nabla U(x_k)-\frac{h}{2}\nabla U(x_{k+1})
\end{align}
With $\alpha=1/2$, we have only two non-vanishing weights, given by
\[W_0=\begin{pmatrix}
    1&1\\
    -1&1
\end{pmatrix},\quad W_1=\begin{pmatrix}
    0&-2\\
    0&0
\end{pmatrix}.\]
Then, the approximation \eqref{RKConQua2} becomes 
\[\mathcal{D}^{(2\alpha)}_{-}{\bf x}_k= \frac{1}{h}\begin{bmatrix}
    x_{k+1}-x_k\\
        x_{k+1}-x_k
\end{bmatrix}\]
Now, we want to define the $qp$-formula, i.e.,~the map $(x_k,p_k)\mapsto (x_{k+1},p_{k+1})$ to \eqref{ContFracDamp:a}. We compute explicitly the discrete Legendre transform as
\begin{align}\label{eq:mumentum}
   \begin{split}
        p_k^-&=\frac{1}{h}\left(x_{k+1}-x_k\right)+\frac{h}{2}\nabla U(x_k) +\frac{\rho}{2}(x_{k+1}-x_k) \\
   p_{k+1}^+&=\frac{1}{h}\left(x_{k+1}-x_k\right)-\frac{h}{2}\nabla U(x_{k+1}) -\frac{\rho}{2}(x_{k+1}-x_k).
   \end{split}
\end{align}
\begin{theorem}
Let $\alpha=1/2$ with RKCQ associated to the $2$-stage Lobatto IIIC. The FVI with the corresponding scheme \eqref{eq:mumentum} applied to the system \eqref{eq:xpart_frac_hamiltonian_sys} is a second-order accurate method for both $x$ and $p$ variables. This means that their local errors are proportional to $O(h^3)$.
\end{theorem}

\begin{proof}
By a simple computation, solving the two equations in \eqref{eq:mumentum} with respect to $x_{k+1}$ and $p_{k+1}$, respectively, gives
\begin{align}
    x_{k+1}&=x_k + \frac{2h}{2+\rho h}\left(p_k-\frac{h}{2}\nabla U(x_{k}) \right) \\
   p_{k+1}&=\frac{2-\rho h}{2+\rho h}\left(p_k-\frac{h}{2}\nabla U(x_k)\right)-\frac{h}{2}\nabla U(x_{k+1}), 
\end{align}
 Expand, for a small $h$, in Taylor series the exact position and momentum  $x_{\text{ex}}(t+h)$ and $p_{\text{ex}}(t+h)$ gives 
\begin{align}
   x_{\text{ex}}(t+h)&=x + hp+\frac{h^2}{2}\dot p+\frac{h^3}{6}\ddot p+O(h^4)\nonumber\\
             &=x + hp-\frac{h^2}{2}(\nabla U(x)+\rho p)+\frac{h^3}{6}\ddot p+O(h^4)\label{xexact}\\
   p_{\text{ex}}(t+h)&=p + h\dot p+\frac{h^2}{2}\ddot p+O(h^4)\nonumber\\
   &= p - h (\nabla U(x)+\rho  p)+\frac{h^2}{2}\left(-\frac{d}{dt}\nabla U(x)+\rho \nabla U(x)+\rho^2p\right)+O(h^3),\label{pexact}
\end{align}
here $x=x_{\text{ex}}(t)$ and $p=p_{\text{ex}}(t)$.  Similarly, expanding the numerical update $x_{k+1}$ yields 
\begin{align}
    x_{k+1}&=x_k + \frac{2h}{2+\rho h}\left(p_k-\frac{h}{2}\nabla U(x_{k}) \right)\nonumber \\
&=x_k + h\left(1-\frac{\rho h}{2}+\frac{\rho^2 h^2}{4}+O(h^3)\right)\left(p_k-\frac{h}{2}\nabla U(x_{k}) \right)\nonumber \\    
&=x_k+hp_k-\frac{h^2}{2}(\nabla U(x)+\rho p_k)+ \frac{\rho h^3}{4}(\nabla U(x_k)+\rho p_k)+O(h^4).\label{xnum}
\end{align}
For the update $p_{k+1}$, we have for a small $h$ that $\frac{2-\rho h}{2+\rho h}=1-\rho h+\rho^2 h^2/2+ O(h^3)$ and $\nabla U(x_{k+1})=U(x_{k})+h\frac{d}{dt}\nabla U(x_k)+O(h^2)$, so that
\begin{align}
 p_{k+1}&=\left(1-\rho h+\rho^2 h^2/2+ O(h^3)\right)\left(p_k-\frac{h}{2}\nabla U(x_k)\right)-\frac{h}{2}\left(U(x_{k})+h\frac{d}{dt}\nabla U(x_k)+O(h^2)\right)\nonumber\\
 &=p_k-h\left(\nabla U(x_k)-\rho p_k\right)+\frac{h^2}{2}\left(-\frac{d}{dt}\nabla U(x_k)+\rho \nabla U(x_k)+\rho^2p_k\right)+O(h^3).\label{pnum}
\end{align}
Comparing \eqref{xexact} with \eqref{xnum} and \eqref{pexact} with \eqref{pnum}, we thus have the local truncation errors for $x$ and $p$, that is
\[\left\| x_{k+1}-x_{\text{ex}}(t_{k+1})\right\|=O(h^3),\quad \left\| p_{k+1}-p_{\text{ex}}(t_{k+1})\right\|=O(h^3).\]
Hence,  the global error is $O(h^2)$ for  both $x$ and $p$.
\end{proof}

\subsection{Numerical Experiment}
We consider the following bidimensional problem associated to the coupled oscillators system subject to fractional damping
\begin{subequations}\label{eq:damped-oscillator}
\begin{align}
    m_1\ddot x_1 +\eta_1x_1  =- \gamma_1\,D_-^{(2\alpha)} x_1, \\
   m_2\ddot x_1 +\eta_2x_2   =- \gamma_2\,D_-^{(2\alpha)}x_2.
\end{align}
\end{subequations}
The corresponding conservative Lagrangian of this model is given by 
$$L(x,\dot x)=\frac{1}{2}\left(\dot x_1^2+\dot x_2^2\right)-\frac{1}{2} \left(\eta_1x_1^2+\eta_2x_2^2\right),$$ 
with $x=(x_1,x_2)$.
In these simulations, we consider the forced coupled oscillator, i.e.,~$\alpha=\frac{1}{2}$ subject to the initial conditions $x=(0.8,-0.5),\ \dot x =(0.4,0.0)$ with $m_1=m_2=1$, $\eta=\eta_1=\eta_2=0.5$ and damping coefficients  $\gamma_1=\gamma_2=0.25$, over the time interval $[0,20]$ with a time step $h=0.2$. Figure \ref{fig:forced_oscillator} top, shows the numerical solutions, while the energy decay is presented in the bottom left figure.  Given the energy of the system which defined by the Hamiltonian $E(t)=H(x,p)=\frac{1}{2}\|p\|^2+\frac{\eta}{2}\|x\|^2$ with $x=(x_1,x_2)$ and $p=(p_1,p_2)$, the bottom right figure displays the relative error of the energy in logarithmic scale, computed  as
\[E_{\text{err}}(t_k)=\frac{E_k-E(t_k)}{\max_{t\in[0,T]}|E(t)|}.\]
where $E_k=H(x_k,p_k)$, i.e.~the energy is evaluated at the approximated solution $x_k=(x_{1,k},x_{2,k})$ and $p_k=(p_{1,k},p_{2,k})$, and $E(t_k)=H(x_{\text{ex}}(t_k),p_{\text{ex}}(t_k))$.
\begin{figure}[htbp]
\centering
\includegraphics[width=.6\textwidth]{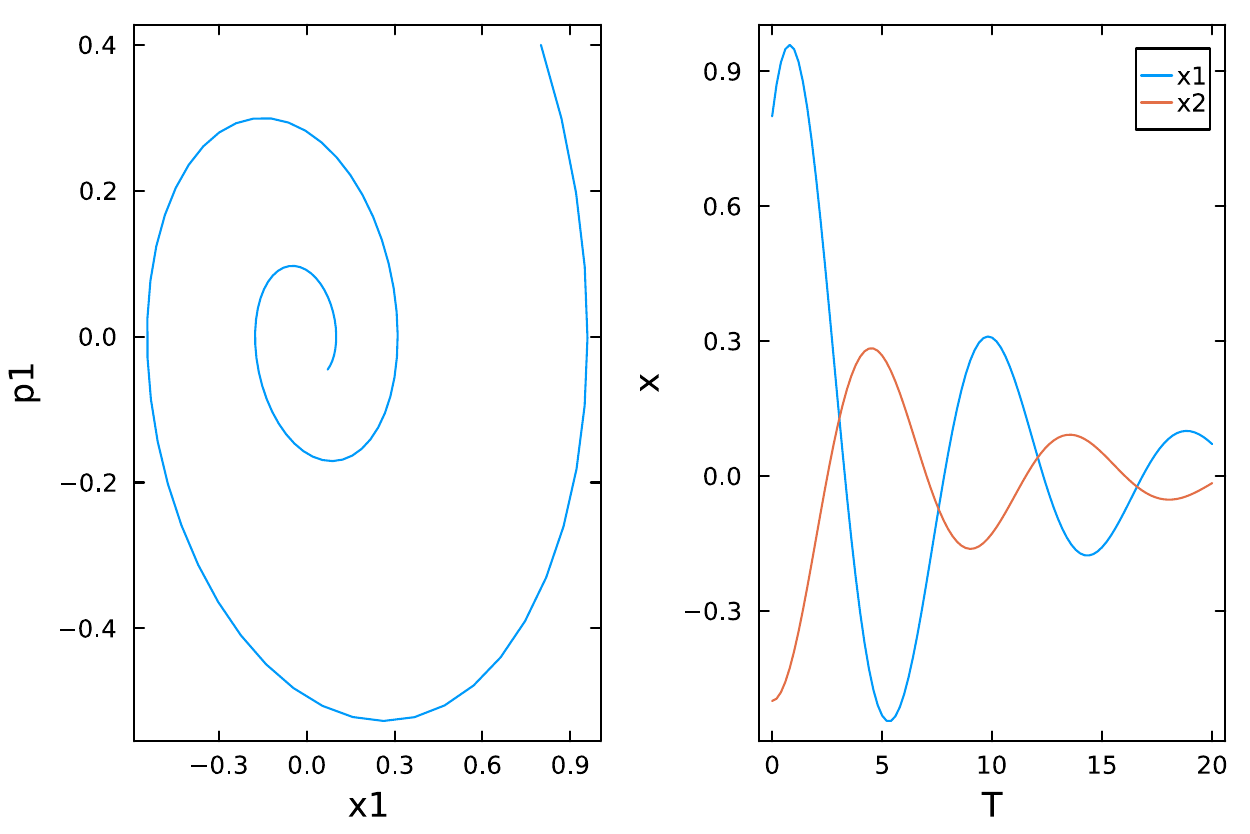}\\
\includegraphics[width=.6\textwidth]{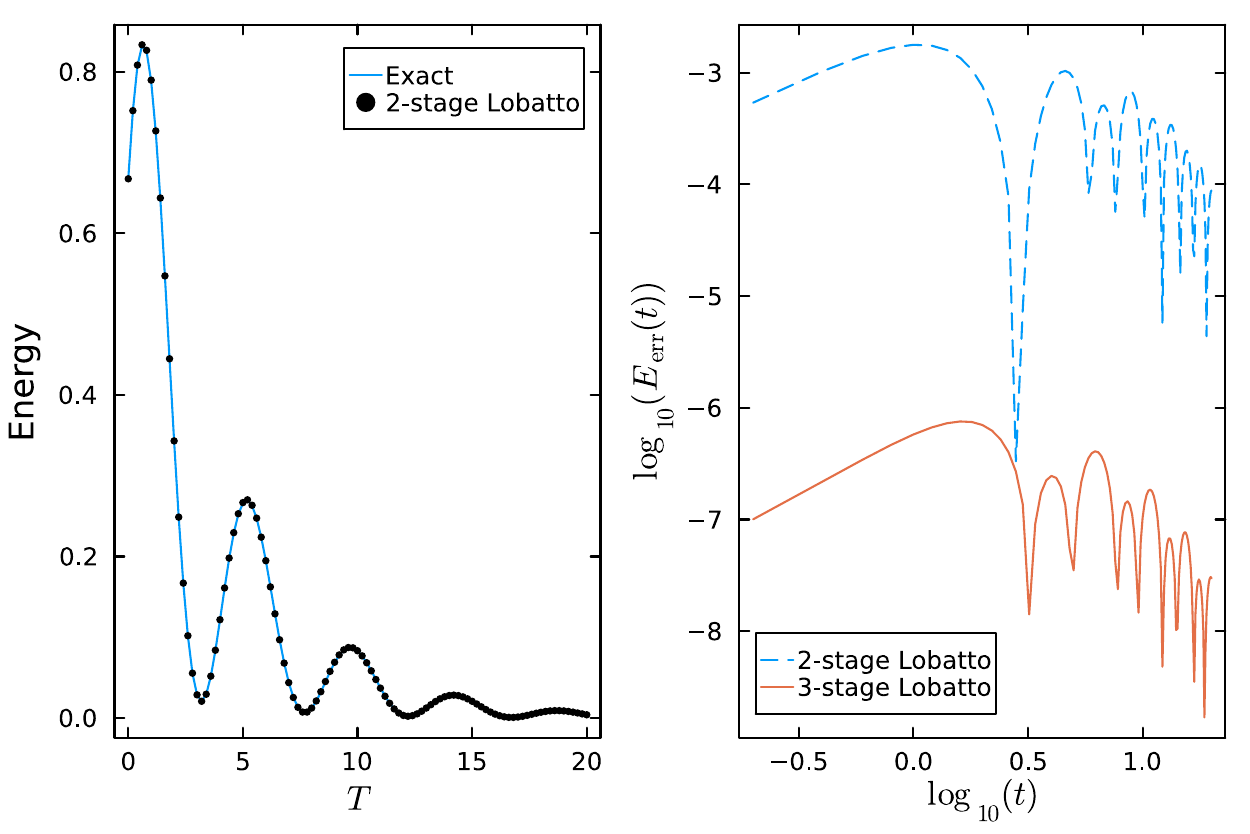}
\caption{Numerical results obtained over the interval $[0,20]$ with a step size
$h=0.2$ for bidimensional forced oscillator \eqref{eq:damped-oscillator} $(\alpha=1/2)$.
\textit{Top:} Numerical trajectories using 2-stage Lobatto IIIC.
\textit{Bottom left:} Energy decay. \textit{Bottom right:} Energy relative error $E_{\text{err}}(t_i)$.}
\label{fig:forced_oscillator}
\end{figure}

Figure \ref{fig:forced_oscillator} shows that the energy decay from our integrator is in excellent agreement with the exact solution. On the other hand, the relative variations in energy are very small, confirming that this physical quantity is very well approximated by our numerical method.

Now, we test the efficiency of FVI scheme \eqref{FinalRK1}-\eqref{FinalRKi} based on  $r$-stage Lobatto IIIC  with $r=2,3,4$. Let $(x(t) , p(t))$ be the exact position-momentum solution of the Hamiltonian system associated to \eqref{eq:damped-oscillator}. The errors at the main nodes are then computed as the global errors
\[\max_{\substack{k\in\{0,\ldots,N\} \\ i\in\{1,2\}}} \left|x_k - x(t_k)\right|,\qquad\max_{\substack{k\in\{0,\ldots,N\} \\ i\in\{1,2\}}} \left|p_k - p(t_k)\right|.
 \]
Figure \ref{fig:FVIcovergence-integer} confirms that, the integrator $r$-stage Lobatto IIIC  with $r=2,3,4$ is of order $O(h^{2r-2})$. This result coincide with the one obtained by \cite{SinaVerm} for the discrete Lagrange-d'Alembert principle.

\begin{figure}[htbp]
\centering
\includegraphics[width=.6\textwidth]{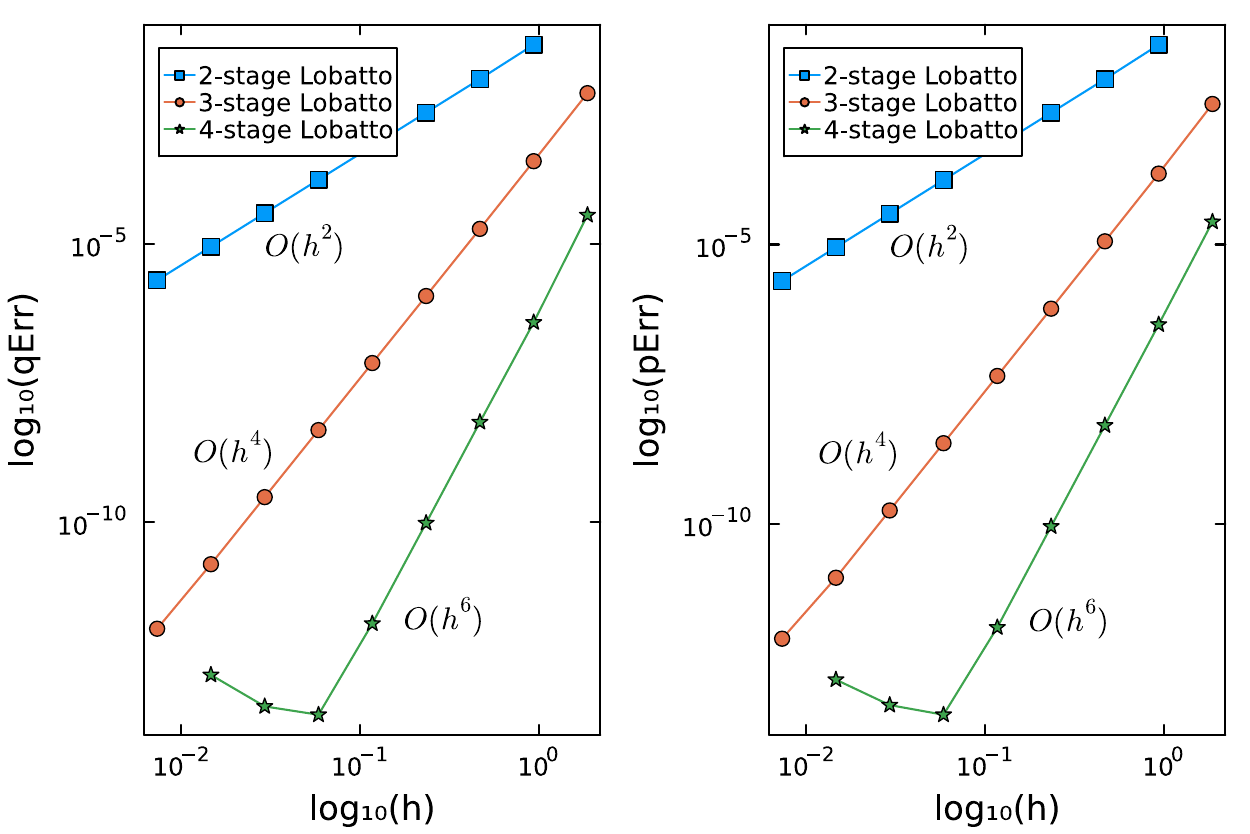}
\caption{Bidimensional forced oscillator \eqref{eq:damped-oscillator}. The accuracy of FVI based on Lobatto IIIC CQ with respect to position and momentum, Algorithm \ref{alg:FractionalAlgorithm} with $t\in[0,30]$ and $N=2^i,\ i=5,\ldots,12$.}
\label{fig:FVIcovergence-integer}
\end{figure}

In the following, we apply our FVI scheme with Lobatto IIIC CQ  to the following  one-dimensional fractional damping model 
\begin{equation}
\ddot x + \rho\ D_-^{(2\alpha)} x + x =f(t),\quad x(0)=\dot x(0)=0.  \label{eq:Torfik}
\end{equation}
This equation belongs to the family of  Bagley–Torvik equations and it can be derived via the restricted variational principle with a time dependent Lagrangian of the form $L(t,x,v)=v^2/2-x^2/2+xf(t)$. In addition, it is well known that analytical solutions of such problems are generally difficult to obtain and,  when available, they are often expressed in terms of special functions, i.g.,~the Mittag-Leffler function, making their evaluation computationally expensive. Nevertheless, for our specific problem with $\rho =1$, $\alpha=0.5$ and $f(t)=t^3+6t+3.2t^{2.5}/\Gamma(0.5)$ ($\Gamma$ is the Gamma function), the exact solution exists and is given explicitly by $x(t)=t^3$, see \cite{Ford}.

\begin{figure}[htbp]
\centering
\includegraphics[width=.6\textwidth]{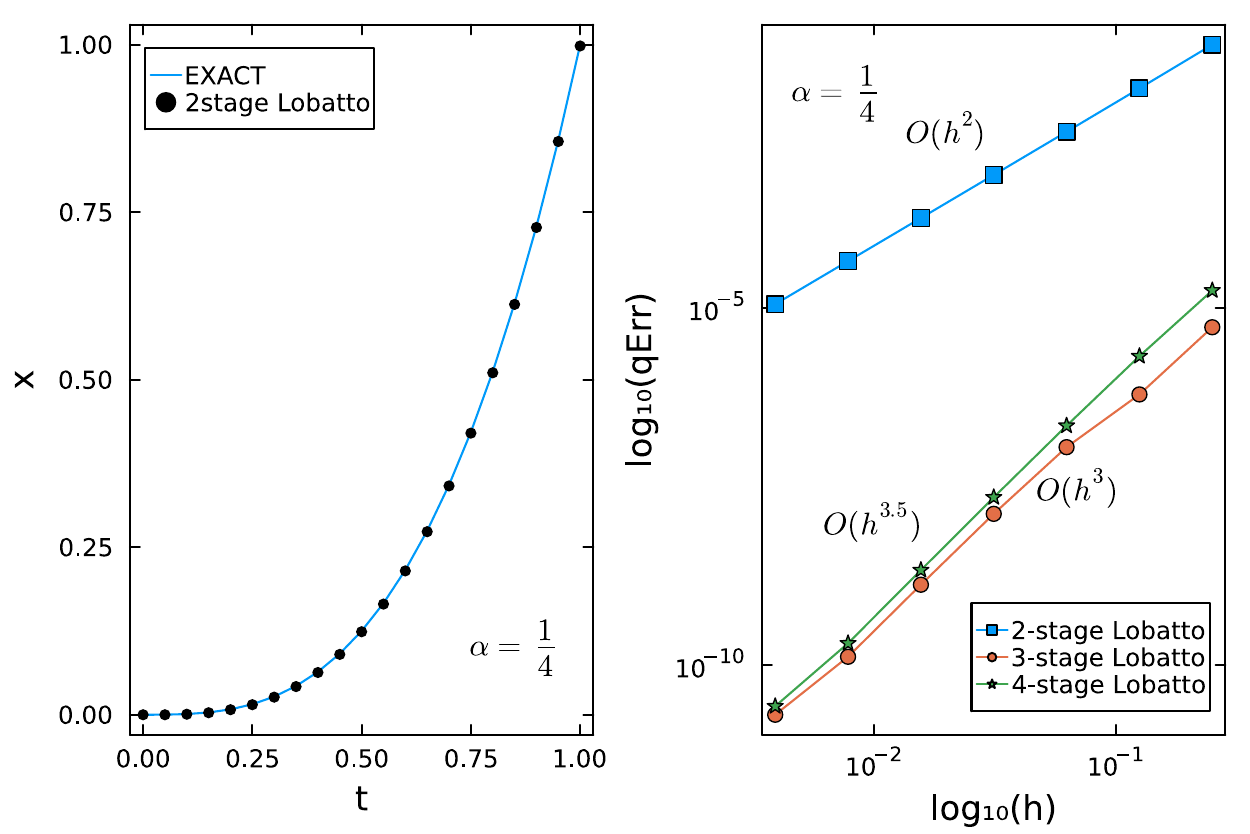}
\caption{Fractional Bagley–Torvik equation \eqref{eq:Torfik} with $\alpha=1/4$, i.e.,~involves the half-derivative. \textit{Left:} Numerical trajectory. \textit{Right:} The accuracy of Lobatto FVI, Algorithm \ref{alg:FractionalAlgorithm2}
 with $t \in [0, 1]$ and $N = 2^i,\ i = 2, \ldots, 8.$}
\label{fig:covergence-05}
\end{figure}

We compute the error term as
\[\max_{\substack{k\in\{0,\ldots,N\}}} \left|x_k - x(t_k)\right|.\]

As illustrated in Figure \ref{fig:covergence-05}, the proposed scheme demonstrates good agreement with the exact solutions. The observed convergence rates of  $O(h^{2})$, $O(h^{3})$ and $O(h^{3.5})$ align well with the predictions of Theorem \ref{thm:RKCQ_convergence}, with only minor discrepancies attributable to numerical errors

\subsection{Note on RKCQ based on Gauss methods}
The analysis of RKCQ based on Gauss methods, when applied to hyperbolic operators, has recently been reported in \cite{LeMa}.
\begin{remark}
The definition of RKCQ \eqref{RKConQua} based on the evaluation of the function $f$ at the RK nodes  $\left\{t_k+c_ih\right\}_{i=1}^r$.  On the other hand, the variations $\delta x_k^i$ appearing in \eqref{eq:variation_proof} also involve the main nodes, i.e.,~with $c_1=0$ and $c_r=1$, which must coincide  with the variations of Lemma \ref{VariationCommRK}. This is precisely why we employ Lobatto IIIC in the discrete action \eqref{DiscFracActionHORK}. For Radau or Gauss methods, the main nodes are not directly included. Thus, one would need to interpolate the function $f$.
 \end{remark}

In particular, with the  midpoint   $\setlength{\tabcolsep}{15pt}
\renewcommand{\arraystretch}{1.2}\begin{array}{c|c}
\frac{1}{2}   &  \frac{1}{2}   \\ \hline
     &  1
\end{array},$
the CQ generating function is  given by
$$\gamma_{\text{mid}}(z)=2\frac{1-z}{1+z},$$
which coincide  the generating function of the trapezoidal rule, a method for which convolution quadrature theory is already well established in \cite{ErSa} and its theoretical order is $2$ at the midpoint stage, see also \cite{LehSaBook} for more details.
We can define a modified midpoint CQ method (MIDCQ) using a local linear interpolation $\tilde f(t)=f_j+(t-t_j)(f_{j+1}-f_j)/h$ on each subinterval $[t_j,t_{j+1}]$, that is
$$\mathcal{D}_{-}^{\alpha} {f}_k\approx \sum_{j=0}^{k}w_{k-j}^{(\alpha)}\tilde f\left(t_j+h/2\right)=\sum_{j=0}^{k}w_{k-j}^{(\alpha)}\left( \frac{f_j+f_{j+1}}{2}\right),$$
where $w_{k-j}^{(\alpha)}$ is the Taylor expansion of $(\gamma_{\text{mid}}(z)/h)^\alpha$. This allows to approximate the fractional derivatives at $t_k+h/2$ and use the variation at the main nodes which are needed to derive the midpoint scheme.
\begin{algorithm}{}
\begin{algorithmic}[1]
\State {\bf Initial data}: $N,\, h,\,\alpha,\,w_n^{(\alpha)},\, x_0,\, p_{0}.$
\State {\bf Solve for} $x_1$ {\bf from} 
\[
p_{0}=-D_1L_d(x_0,x_1)+\rho hw_0^{(\alpha)}\left(\frac{x_1-x_0}{2}\right)
\]
\State {\bf Initial points:} $x_0, x_1$
    \For {$k= 1: N-1$} 

\State  {\bf Solve for} $x_{k+1}$ {\bf from} 
\[
0=D_{2}L_d(x_{k-1},x_{k})+ D_1L_d(x_k,x_{k+1})-\rho h \left(\mathcal{D}^{(2\alpha)}_{-}{\bf x}_k+\mathcal{D}^{(2\alpha)}_{-}{\bf x}_{k-1}\right)
\]
    \EndFor
    \State  {\bf Output:} $(x_2, \ldots,x_{N}).$
\end{algorithmic}
\caption{FVI with MIDCQ}\label{alg:FractionalAlgorithm2}
\end{algorithm}

\subsubsection{Numerical Experiment}
Now we apply the  midpoint scheme MIDCQ to the one-dimensional linear damped oscillator and a fractional damping model  \cite{Torvik} to observe the applicability and accuracy of the developed algorithm. Precisely, we consider the two following models
\begin{align}
\ddot x &+ \rho\ \dot x + x =0,\quad \rho =0.25,\quad x(0)=1,\quad \dot x(0)=0.5 \label{eq:1DdampedOsci} 
\end{align}
The second equation belongs to the family of  Bagley–Torvik equations which can be derived by the restricted variational principle with a time dependent Lagrangian of the form $L(t,x,v)=v^2/2-x^2/2+xf(t)$. The analytical solution is given by $x(t)=t^3$ for $\alpha=0.5$ with $f(t)=t^3+6t+3.2t^{2.5}/\Gamma(0.5)$ ($\Gamma$ is the Gamma function) with the data given above, see \cite{Ford}.

Figures \ref{fig:midpoint-covergence-integer} and \ref{fig:midpoint-covergence-noninteger} show that the FVI  MIDCQ variational integrator, as expected, still preserves full order $2$.

\begin{figure}[htbp]
\centering
\includegraphics[width=.6\textwidth]{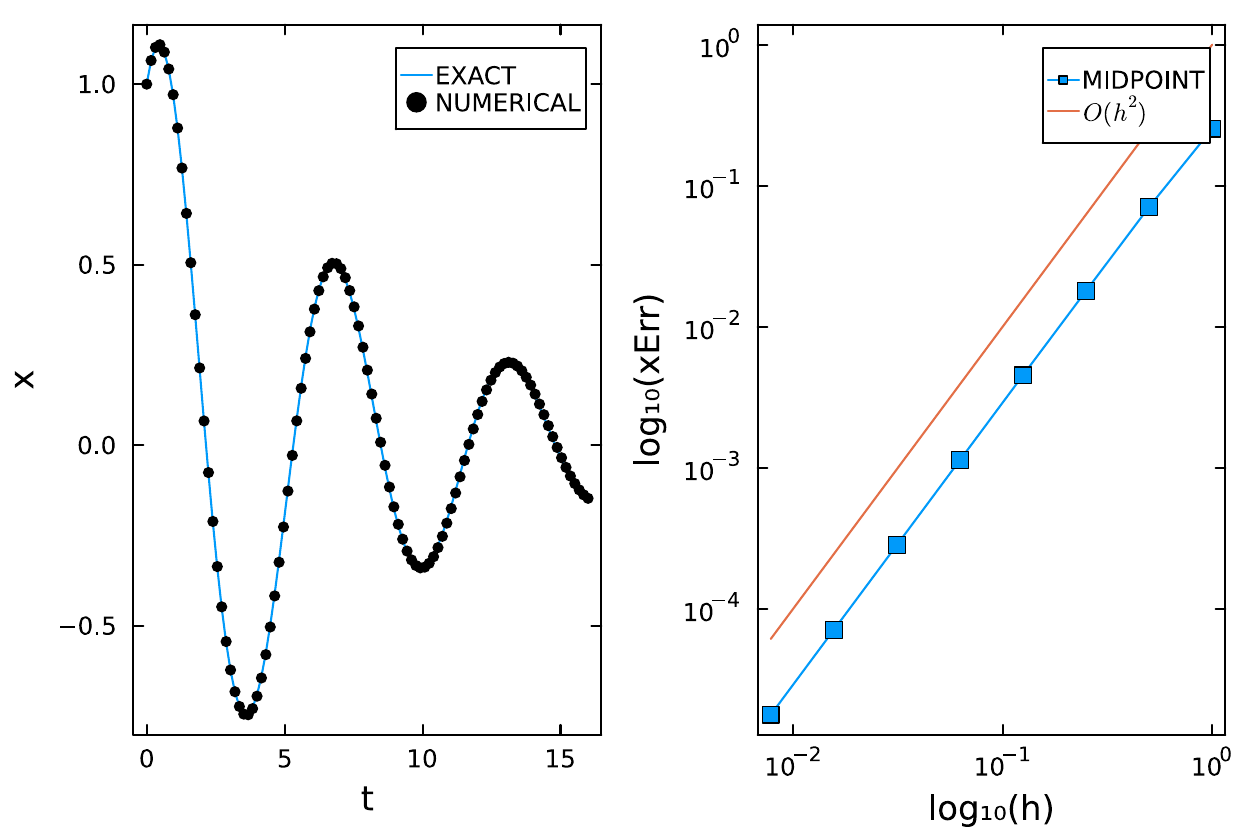}
\caption{1D damped oscillator \eqref{eq:1DdampedOsci}. \textit{Left:} Numerical trajectory. \textit{Right:} The accuracy of FVI with MIDCQ, Algorithm \ref{alg:FractionalAlgorithm2}
 with $t \in [0, 16]$ and $N = 2^i,\ i = 4, \ldots, 11.$}
\label{fig:midpoint-covergence-integer}
\end{figure}

\begin{figure}[htbp]
\centering
\includegraphics[width=.6\textwidth]{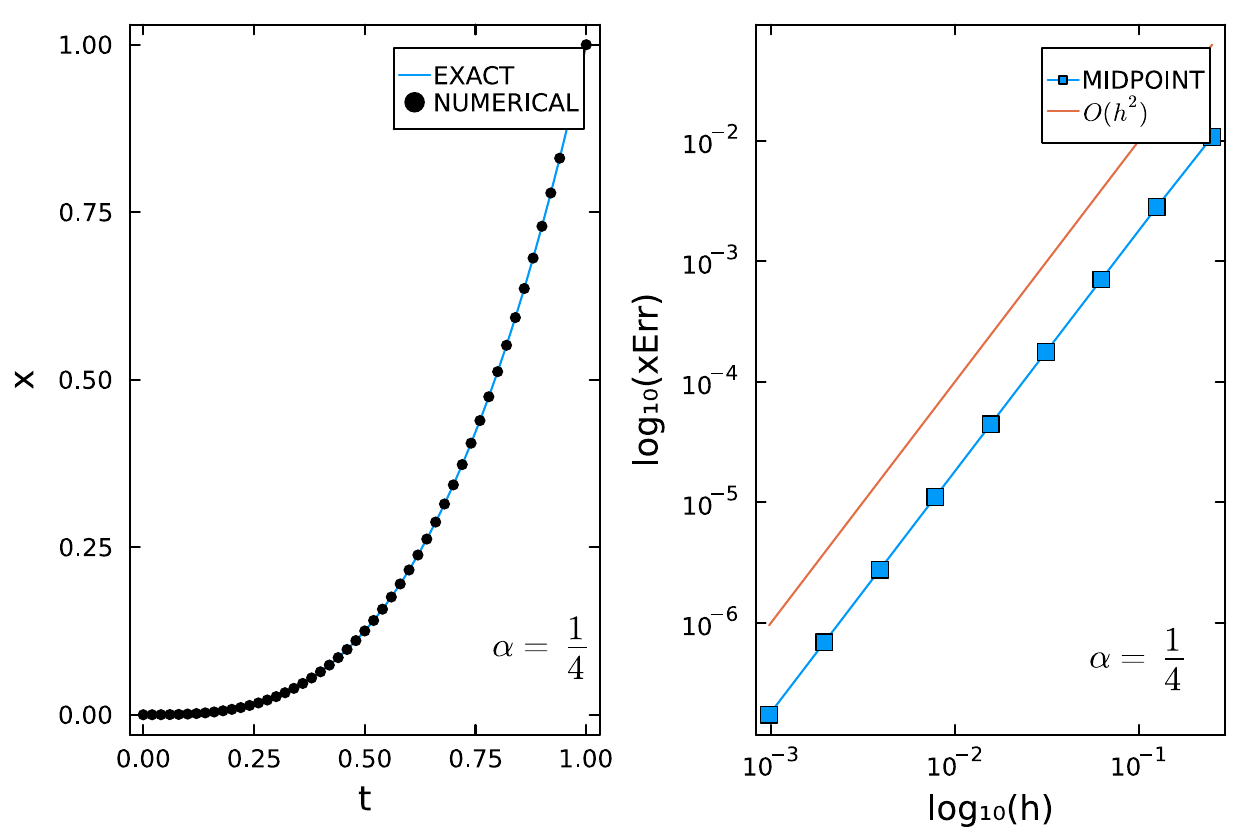} \caption{Fractional Bagley–Torvik equation \eqref{eq:Torfik} $(\alpha=1/4)$. \textit{Left:} Numerical trajectory. \textit{Right:} The accuracy of FVI with MIDCQ, Algorithm \ref{alg:FractionalAlgorithm2}
 with $t \in [0, 1]$ and $N = 2^i,\ i = 4, \ldots, 11.$}
\label{fig:midpoint-covergence-noninteger}
\end{figure}

\section{Conclusion}\label{conclusion}

In this paper, we have proposed a discrete setting of the restricted Hamilton's principle using two different approaches for the continuous actions \eqref{FracAction}. One for conservative part $\Lc$\cite{HaLe13,MaWe,SinaSaake} and the other for the fractional part $\Lf$ by means of a convolution quadrature based on Lobatto IIIC Runge-Kutta methods \cite{LuOs,BaLu,BaLo}. In addition, we developed a midpoint-based fractional variational integrator, which achieves second-order accuracy.

Unlike the backward-difference-based method \cite{HaJiOb}, which is limited to second-order accuracy due to saturation effects, the Runge–Kutta–based FVI avoids limitations and can achieve higher-order accuracy. In the case of $\alpha=1/2$, the theoretical accuracy for Lobatto IIIC coincides with the result given by \cite{SinaVerm}. We also test our scheme for the fractional Bagley–Torvik equation \eqref{eq:Torfik}.

While Lagrangian variational error analysis \cite{MaWe} indicates that the order of variational integrators depends on the accuracy of the discrete Lagrangian, 
a full theoretical analysis of the variational error in the fractional setting remains to be developed. This might not be trivial, since the Lagrangian depends on more than just the tangent bundle variables.

Future work will focus on establishing a rigorous fractional variational error analysis, determining the theoretical convergence order for fractional Bagley–Torvik equation, and developing fractional variational integrators based on Gauss-type RKCQ.

\printbibliography
\end{document}